\documentclass[12pt]{amsart}

\usepackage{vmargin,graphicx,amsmath,amssymb,color,subfigure,enumerate,csquotes,epstopdf}
\usepackage[utf8]{inputenc}
\usepackage{hyperref}
\usepackage{verbatim}

\newtheorem{theorem}{Theorem}[section]
\newtheorem{lemma}[theorem]{Lemma}
\newtheorem{proposition}[theorem]{Proposition}

\theoremstyle{definition}
\newtheorem{definition}[theorem]{Definition\rm}
\newtheorem{notation}[theorem]{Notation}
\newtheorem{remark}[theorem]{Remark}
\newtheorem{assumption}[theorem]{Assumption}

\makeatletter

\@addtoreset{equation}{section}
\makeatother

\newcommand{\NN}{\mathbb{N}}
\newcommand{\R}{\mathbb{R}}

\renewcommand{\S}{{\mathbb{T}}}

\newcommand{\dx}{\mathrm{d}}
\renewcommand{\L}{\mathcal{L}_{\eps,\alpha}}
\newcommand{\T}{{\sf T}_{\phi_\eps}}
\newcommand{\LL}{\mathfrak{L}}

\renewcommand{\O}{\mathcal{O}}
\newcommand{\Q}{\mathcal{Q}_{\eps,\alpha}}

\renewcommand{\th}{\rm th}
\newcommand{\x}{\mathsf{x}}
\newcommand{\y}{\mathsf{y}}
\newcommand{\ess}{\mathsf{ess}}

\newcommand{\eff}{\mathsf{eff}}
\newcommand{\app}{\mathsf{app}}
\newcommand{\red}{\mathsf{red}}
\newcommand{\cri}{\mathsf{cr}}
\newcommand{\sa}{\mathsf{s.a.}}
\newcommand{\spec}{\mathsf{sp}}
\renewcommand{\sc}{\mathsf{s.c.}}
\newcommand{\eps}{\varepsilon}

\renewcommand{\paragraph}[1]{\medskip\noindent{\bf #1}}

\usepackage{mathtools} 
\DeclarePairedDelimiter\abs{\lvert}{\rvert}

\DeclarePairedDelimiter\Norm{\big\lVert}{\big\rVert}

\title[Spectral asymptotics for the Schr\"odinger operator ]{Spectral asymptotics for the Schr\"odinger operator on the line with spreading and oscillating potentials}

\author[V. Duch\^ene]{V. Duch\^ene}
\address[V. Duch\^ene]{IRMAR, CNRS, Universit\'e de Rennes 1, Campus de Beaulieu, F-35042 Rennes cedex, France}
\email{vincent.duchene@univ-rennes1.fr}

\author[N. Raymond]{N. Raymond}
\address[N. Raymond]{IRMAR, Universit\'e de Rennes 1, Campus de Beaulieu, F-35042 Rennes cedex, France}
\email{nicolas.raymond@univ-rennes1.fr}

\begin{document}
\maketitle
\begin{abstract}
This study is devoted to the asymptotic spectral analysis of multiscale Schr\"odinger operators with oscillating and decaying electric potentials. Different regimes, related to scaling considerations, are distinguished. By means of a normal form filtrating most of the oscillations, a reduction to a non-oscillating effective Hamiltonian is performed.
\end{abstract}
\tableofcontents
\newpage
\section{The problem}

\subsection{Context and motivation}
In this work, we study the asymptotic behavior (as $\epsilon \to 0$) of the low-lying spectrum of the following multiscale Schr\"odinger operator on the line:
\begin{equation} \label{Pb.initial-beta}  
\mathcal L_{\epsilon,\beta} :=  \epsilon^\beta D_{\x}^2+q\left({\x},\frac{\x}{\epsilon}\right)\,, \quad D_\x=\frac1i\frac{\dx}{\dx \x}\,,
\end{equation}
where $q(x,y)$ is localised in the first variable, and $1$-periodic and \emph{zero-mean} in the second variable. In other words, we are interested in the non-zero solutions of the following eigenvalue equation:
\begin{equation} \label{Pb.initial-beta'}  
\mathcal L_{\epsilon,\beta}\varphi_{\epsilon,\beta}({\x}) = \lambda_{\epsilon,\beta}\varphi_{\epsilon,\beta}({\x})\,,\qquad \varphi_{\epsilon,\beta}\in L^2(\mathbb{R})\,.
\end{equation}
In addition to its intrinsic mathematical interest, the spectral investigation of \eqref{Pb.initial-beta} can be motivated as a toy problem for the propagation of waves in a material with \emph{high contrast microstructure} (see for instance \cite{ACPSV04}). 

Of course, in the case of the trivial potential $q\equiv 0$, the spectrum is purely essential, and no eigenfunction with finite energy is allowed. Adding a spatially localized potential $q\left(\mathsf{x},\frac{\mathsf{x}}{\epsilon}\right)$ does not perturb the essential spectrum~\cite{RS4}, so that $\spec_{\ess}(\mathcal L_{\epsilon,\beta})=\R^+$. However, as we shall see, the presence of the highly oscillatory potential generates negative eigenvalues. Our aim is to describe the asymptotic behavior of these eigenvalues through non-oscillatory, effective operators of the form:
\[\mathcal L^{\eff}_{\epsilon,\beta}:= \epsilon^\beta D_{\x}^2 +\epsilon^{\gamma_0} V_0({\x}) +\epsilon^{\gamma_1} V_1({\x})  \quad \text{ on } L^2(\R)\,,\]
where $V_0$ and $V_1$ are given in terms of $q$, and independent of $\epsilon$, and $\gamma_0(\beta),\gamma_1(\beta)\in\R$.

The eigenvalue asymptotics may be very different, depending on the value of the parameter $\beta$, as one can see by looking at the following two cases which have been treated in the literature.
\begin{enumerate}[\rm i.]
\item The case $\beta=0$ corresponds to the classical case of \emph{homogenization}. Although standard homogenization arguments yield an effective potential 
\[V_{\eff}({\x})=\int_\S q({\x},{\y})\dx {\y}=0\,,\] 
a more precise study~\cite{BorisovGadylshin06,DucheneVukicevicWeinstein14} shows that the low-lying spectrum is driven by a non-trivial effective potential $V_{\eff}({\x})=\epsilon^2 V({\x})\leq 0$. Consistently, there exists a negative eigenvalue, $\lambda_{\epsilon,0}\sim -\epsilon^4\frac14\big(\int_\R V\big)^2$. This eigenvalue is unique for $\epsilon$ sufficiently small, and the corresponding eigenfunction behaves like $\varphi_{\epsilon,0}\sim \exp(-\sqrt{-\lambda_{\epsilon,0}}|\cdot|)$.

\item The case $\beta=2$ has been studied in~\cite{Dimassi15,DimassiDuong17}, and corresponds to a \emph{semiclassical} scaling. In particular, Dimassi shows that the number of negative eigenvalues grows as $\epsilon\to 0$ and satisfies a Weyl type asymptotics. Although the method therein relies on the use of an effective Hamiltonian, it is not clear how to relate this effective Hamiltonian to an effective potential, even in our one-dimensional setting.
\end{enumerate}

\subsection{Results}
Here we aim at studying the situation of intermediate values for $\beta$. Our results apply to $\beta\in (0,3/2)$ and predict very different asymptotic behaviors (having the flavor of situation i. or ii. above) depending on the sign of $\beta-1$. It is interesting to notice that our strategy does not depend strongly on the situation at stake. We find it convenient to study a rescaled problem: let $\alpha=\frac{\beta}{2-\beta}$, $\eps=\epsilon^{1-\beta/2}$ and (abusing notations) $\varphi_{\eps,\alpha}(x)=\varphi_{\epsilon,\beta}(\epsilon^{\beta/2}x)$, $\lambda_{\eps,\alpha}=\lambda_{\epsilon,\beta}$. Then the eigenproblem~\eqref{Pb.initial-beta'} reads
\begin{equation} \label{Pb.initial}
 \L\varphi_{\eps,\alpha}(x) := \left( D_x^2+q(\eps^\alpha x,x/\eps)\right)\varphi_{\eps,\alpha}(x) \ = \ \lambda_{\eps,\alpha} \varphi_{\eps,\alpha}(x), \quad \varphi_{\eps,\alpha}\in L^2(\mathbb{R}).
 \end{equation}
From now on, we shall only focus on the eigenproblem~\eqref{Pb.initial}; the interested reader can straightforwardly translate our results to the original problem~\eqref{Pb.initial-beta'}. 

\subsubsection{An effective Hamiltonian to describe the low-lying spectrum}
Let us describe the explicit formula for the aforementioned effective potential. A key role is played by the function $Q$, defined as the unique solution to
\begin{equation}\label{eq.cell}
D_y^2Q(X,y)=-q(X,y), \quad \int_\S Q(X,y)\dx y=0.
\end{equation}
Here, $X\in\R$ is a fixed parameter. This allows to define
\begin{equation}\label{eq.V}
V_0(X)=-\int_\S \abs{\partial_y Q}^2(X,y)\dx y\ \quad \text{ and } \quad V_1(X)  = 2 \int_\S \big((\partial_X Q)(\partial_y Q)\big)(X,y)\dx y\,.
\end{equation}
Our aim is to show that $V_0$ and its corrector $V_1$ act as {\em effective potentials}, in the sense that the asymptotic behavior of the low-lying spectrum of our original operator, $ \L$, may be described at first order through the one of the non-oscillatory effective operator:
\begin{equation} \label{Pb.effective}
 \L^{\eff} :=D_x^2+\eps^2 V_0(\eps^\alpha x)+\eps^{3+\alpha} V_1(\eps^\alpha x) \quad \text{ on } L^2(\R)\,.
 \end{equation}
 What is more, in most situations, one does not lose precision by considering 
 \begin{equation} \label{Pb.effective0}
  \L^{\eff,0} :=D_x^2+\eps^2 V_0(\eps^\alpha x)\quad \text{ on } L^2(\R)\,.
  \end{equation}
  
     \begin{notation}
 \rm We use standard notations for Lebesgue spaces, $L^p$, $L^2$-based Sobolev spaces, $H^k$, and denote $W^{\ell,\infty}(\R\times \S):= W^{\ell,\infty}(\R;L^\infty(\S)) $, endowed with the norm
     \[\Norm{q}_{W^{\ell,\infty}(\R\times\S)}=\sup_{l\in\{0,\dots,\ell\}}\Norm{\partial_X^l q}_{L^{\infty}(\R\times\S)}\,.\]
 We also denote $\langle \cdot\rangle:=(1+|\cdot|^2)^{1/2}$.
     \end{notation}
 In the \emph{whole} paper, we work under the following assumption. Additional restrictions are explicitly stated when needed.
 \begin{assumption}
$q\in W^{5,\infty}(\R;L^\infty(\S))$, $\int_{\S} q(X,y)\dx y=0$, $\langle \cdot\rangle V_0'\in L^\infty(\R)$, $\langle \cdot\rangle V_1'\in L^\infty(\R)$ and $V(X)\to 0$ as ${|X|\to\infty}$.
 \end{assumption}

\begin{notation}\label{def.notation}
\rm Denote $\lambda_{n,\eps,\alpha}$ the $n^{\th}$ eigenvalue\footnote{Due to the one-dimensional framework, this eigenvalue is necessarily simple. Indeed, two square-integrable solutions of the eigenvalue problem are necessarily homothetic, since their Wronskian vanishes identically.} of $\L$ counted increasingly, and by convention $\lambda_{n,\eps,\alpha}=0$ if $n>N_{\eps,\alpha}$, the number of negative eigenvalues. Define similarly $\lambda^{\eff}_{n,\eps,\alpha} $ through the eigenvalues of $\L^{\eff}$, and $\lambda^{\eff,0}_{n,\eps,\alpha} $ through the eigenvalues of $\L^{\eff,0}$.
 \end{notation}

We can now state one of our main results. 
\begin{theorem} \label{th.main-result-eigenvalues}
Let $\alpha>-1$ and $q\in W^{5,\infty}(\R\times \S)$ be such that $\langle \cdot\rangle V_0',\langle \cdot\rangle V_1'\in L^\infty$.  There exists $C>0$ such that for all $n\in\NN$ and $\eps\in(0,1]$, one has
 \[  \abs{\lambda_{n,\eps,\alpha} - \lambda_{n,\eps,\alpha}^{\eff} } \leq C \eps^{\min\{4,4(1+\alpha)\}}\,.\]
\end{theorem}

\begin{remark}
\rm The benefit of the estimate of Theorem \ref{th.main-result-eigenvalues} is that the asymptotic behavior of $\lambda_{n,\eps,\alpha}^{\eff}$ for $V$ sufficiently localized, depending on the value of $\alpha$, is well understood. There are three different regimes at stake:
 \begin{enumerate}
 \item[$\alpha>1$] The low-lying spectrum of~\eqref{Pb.effective} is dictated by a \emph{semiclassical limit}. There is a growing number of simple negative eigenvalues accumulating below the edge of the essential spectrum, as $\eps\to 0$;
 \item[$\alpha<1$] The low-lying spectrum of~\eqref{Pb.effective} is dictated by a \emph{weak coupling limit}. Since the effective potential has negative mass, there exists for $\eps$ sufficiently small a unique negative eigenvalue, at a distance $\O(\eps^{4-2\alpha})$ from the origin;
 \item[$\alpha=1$] In this regime, the effective problem is \emph{self-similar}, and there exists a (non-zero) finite number of eigenvalues below the essential spectrum.
  \end{enumerate}
Notice that $\alpha=0$ (respectively $\alpha=+\infty$) corresponds to $\beta=0$ (respectively $\beta=2$), already described, and that the eigenvalue asymptotics are consistent. 
\end{remark}
\begin{remark}
In the situation when $q(X,\cdot)$ is not mean-zero, our method yields the following effective operator
\begin{equation} \label{Pb.effective'}
 \L^{\eff} :=D_x^2+q_{0}(\eps^\alpha x)+\eps^2 \tilde V_0(\eps^\alpha x)+\eps^{3+\alpha} \tilde V_1(\eps^\alpha x) \quad \text{ on } L^2(\R)\,,
 \end{equation}
 where $q_{0}(X)=\int_{\S}q(X,y)\,\dx y$ and $\tilde V_{0}$, $\tilde V_{1}$ are defined as in~\eqref{eq.cell}-\eqref{eq.V}, replacing $q$ by $q-q_{0}$. The nature of the spectrum of~\eqref{Pb.effective'} is generically determined by $q_{0}$, the functions $\tilde V_0$, $\tilde V_{1}$ acting as regular perturbations. This is why we focus on the mean-zero case where the oscillations themselves generate the discrete spectrum.
\end{remark}

\subsubsection{About the approximation of the eigenfunctions}
Theorem \ref{th.main-result-eigenvalues} is valid for a wide range of values for $\alpha$, but in general provides only limited information on the localization of the eigenvalues. However, in the situation where the spectral gap of one of the corresponding operators is asymptotically larger than $\eps^{\min\{4,4(1+\alpha)\}}$, then the asymptotic behavior of $\lambda_{n,\eps,\alpha}$ is described by the one of $\lambda_{n,\eps,\alpha}^{\eff}$. In that case, the effective potential also allows to describe asymptotically the behavior of the corresponding eigenfunctions. This situation generically occurs when $\alpha\in(0,3)$,  as we shall see in the following propositions.

\begin{proposition}[Semiclassical regime]\label{prop.semiclassical}
Assume $\alpha\in(1,3)$. We also assume that $X\mapsto V_0(X)$ has a unique minimum (not attained at infinity) at $X=0$ and that it is non-degenerate. 
Let $N\in\mathbb{N}$. Then there exists $\eps_0>0$, such that if $\eps\in(0,\eps_0)$, then $\L$ has at least $N$ negative eigenvalues, $\lambda_{1,\eps,\alpha}<\dots<\lambda_{N,\eps,\alpha}$,
satisfying
\[
\lambda_{n,\eps,\alpha}\ = \ \eps^2 V_0(0) +\eps^{1+\alpha}(2n-1)\sqrt{\frac{V_0''(0)}2} \ + \ \O( \eps^{\min\{4,2\alpha\}})\,.
\]
Up to changing its sign, the corresponding $n^{\th}$ $L^2$-normalized eigenfunction, $\psi_{n,\eps,\alpha}$, satisfies
 \[\Norm{\psi_{n,\eps,\alpha} - \varphi_{n,\eps,\alpha}^{\eff,0} }_{L^2(\R)} =\O(\eps^{3-\alpha})\,,\]
where $\varphi_{n,\eps,\alpha}^{\eff,0}$ is the $n^{\th}$ $L^2$-normalized eigenfunction of the effective operator $\L^{\eff,0}$ defined in~\eqref{Pb.effective0}. Moreover, we have the approximation
\[\Norm{\varphi^{\eff,0}_{n,\eps,\alpha}(x)-\eps^{\frac{1+\alpha}4} H_n(\eps^{\frac{1+\alpha}2}x)}_{L^2(\R)}=\mathcal{O}(\eps^{\frac{\alpha-1}{2}})\,,\]
where $H_n$ is the $n$-th rescaled Hermite function satisfying
\[-H''_{n}+\frac{V_0''(0)}{2}x^2 H_{n}=(2n-1)\sqrt{\frac{V_0''(0)}{2}}H_{n}\,.\]
If it exists, any other negative eigenvalue satisfies $\tilde\lambda_{\eps,\alpha}\geq \eps^2 V_0(0) +\eps^{1+\alpha}(2N)\sqrt{\frac12V_0''(0)} $.
\end{proposition}
\begin{remark}
The asymptotic behaviour of the eigenvalues is determined by the non-degenerate nature of the minimum of $V_{0}$. Our method could be extended to degenerate situations restricting the range of admissible $\alpha$. 
\end{remark}

\begin{proposition}[Weak coupling regime]\label{prop.small-amplitude}
Let $\alpha\in(0,1)$ and assume that $V_0$ is not almost everywhere zero and satisfies $(1+|\cdot|)V_0,(1+|\cdot|)V_1\in L^1(\R)$. Then there exists $\eps_0>0$ such that for any $\eps\in(0,\eps_0)$, $\L$ has one negative eigenvalue, $\lambda_{\eps,\alpha}$, satisfying
\[ \lambda_{\eps,\alpha}=-\frac14\eps^{4-2\alpha}\left(\int_\R V_0\right)^2+\O(\eps^{\min\{4,6-4\alpha\}})\,,\]
with $L^2$-normalized corresponding eigenfunction satisfying
\[
\Norm{\psi_{\eps,\alpha}(x) -\varphi^\eff_{\eps,\alpha}(x)}_{L^2(\R)}= \O(\eps^{2\alpha})\,,
\]
where $\varphi_{\eps,\alpha}^{\eff}$ is the unique $L^2$-normalized eigenfunction of the effective operator $\L^{\eff}$ defined \eqref{Pb.effective}. Moreover, we have the approximation
\[\left\|\varphi^\eff_{\eps,\alpha}(x)- \left(\frac{\eps^{2-\alpha}}2 \int_{\R}|V_0|\right)^{\frac{1}{2}}\exp \big(|x|\frac{\eps^{2-\alpha}}2\int_{\R}V_0\big)\right\|_{L^2(\R)}=\mathcal{O}(\eps^{\min\{\frac{4}{3}(1-\alpha),1+\alpha\}})\,.\]
If it exists, any other negative eigenvalue satisfies $\tilde\lambda_{\eps,\alpha}=\O(\eps^{4})$.
\end{proposition}
\begin{proposition}[Critical regime]\label{prop.critical}
Let $\alpha=1$ and assume that $V_0$ is not almost everywhere zero and satisfies $(1+|\cdot|)V_0\in L^1(\R)$.
For $n=1,\dots,N$, denote $\lambda_{n,V_0}$ the $n^{\th}$ negative eigenvalue of $D_x^2+V_0$, and $\varphi_{n,V_0}$ its corresponding $L^2$-normalized eigenfunction. Then there exists  $\eps_0>0$ such that for any $\eps\in(0,\eps_0)$, $\L$ has at least $N$ negative eigenvalues, $\lambda_{n,\eps}$, satisfying
\[ \lambda_{n,\eps}=\eps^{2}\lambda_{n,V_0}+\O(\eps^{4})\,,\]
with $L^2$-normalized corresponding eigenfunction satisfying
\[ \Norm{\psi_{n,\eps}(x) - \eps^{1/2}\varphi_{n,V_0}(\eps x)}_{L^2(\R)} = \O(\eps^{2})\,.\]
If it exists, any other negative eigenvalue satisfies $\tilde\lambda_{\eps}=\O(\eps^4)$.
\end{proposition}

\subsubsection{An asymptotic expansion} A natural question one can ask is whether the approximations displayed in our results are the first terms of an asymptotic expansion, at least for smooth potential, $q\in\bigcap_{\ell\in \NN}W^{\ell,\infty}(\R\times\S)$. We are able to positively answer this question only for a limited number of values for $\alpha$, all belonging in the semiclassical regime (see however~\cite{Drouot15} an asymptotic expansion of the eigenvalue in the situation $\alpha=0$). We state below the result for $\alpha=2$.

\begin{proposition}\label{prop.WKB}
Let $\alpha=2$ and $q\in\bigcap_{\ell\in \NN}W^{\ell,\infty}(\R\times\S)$ be such that  $X\mapsto V_0(X)$ has a unique minimum (not attained at infinity) at $X=0$ and that it is non-degenerate. Then $\lambda_{n,\eps,2}$ and $\psi_{n,\eps,2}$, defined by Proposition~\ref{prop.semiclassical}, satisfy expansions in the form of asymptotic series
\[\displaystyle{\lambda_{n,\eps,2}\underset{\eps\to 0}{\sim}\eps^2\sum_{j=0}^{+\infty}\eps^j\lambda_{n,j}}\,,\quad\psi_{n,\eps,2}(x)\underset{\eps\to 0}{\sim}\chi(\eps^2x)e^{-\Phi(\eps^2 x)/\eps}\sum_{j=0}^{+\infty} \eps^j\Psi_{n,j}(\eps^2 x,x/\eps)\,,\]
where
\begin{enumerate}[\rm i.]
\item the last expansion is meant in the $L^2(\R)$-sense,
\item $\lambda_{n,j}\in\R$, $\Psi_{n,j}\in \bigcap_{\ell\in \NN}W^{\ell,\infty}(\R\times\S)$ are defined in Section~\ref{sec.WKB},
\item $\chi$ is a smooth cutoff function equal to $1$ near $0$ and
\[ \Phi(X)=\left|\int_0^X(V_{0}(s)-V_{0}(0))^{\frac{1}{2}} \dx s\right| \underset{X\to 0}\sim \frac12\sqrt{\frac{V_0''(0)}{2}}X^2\,.\]
\end{enumerate}
Moreover, there exists $C_0\in\R$ such that for any $\mathcal U$ bounded neighborhood of $0$,
\[\sup_{\eps^{3/2} x\in \mathcal U} \left|\psi_{n,\eps,2}(x) - C_0 \varphi_{n,\eps}(\eps^{3/2}x)\big(1+\eps^2 Q(\eps^2 x,x/\eps)\big)\right|\lesssim \eps^3\,,\]
and $\varphi_{n,\eps}$ satisfies for any $k\in\NN$, 
\[ \Norm{\varphi_{n,\eps}-H_n}_{C^k(\mathcal U)} \lesssim \eps^{1/2}\,.\]
\end{proposition}

\subsection{Numerical illustration}

In Figures~\ref{fig.alpha=2},~\ref{fig.alpha=05} and~\ref{fig.alpha=1}, below, we plot eigenfunctions of $\L$ in the three different regimes, and compare them with the approximations involved in our results, {\em i.e.} the eigenfunction of the effective operator, $\L^\eff$, and its asymptotic approximation ---namely $\varphi^\app_{n,\eps,\alpha}:=\eps^{\frac{1+\alpha}4} H_n(\eps^{\frac{1+\alpha}2}\cdot)$ in the semiclassical regime, and $\varphi^\app_{\eps,\alpha}:= \big(\frac{\eps^{2-\alpha}}2 \int_{\R}|V_0|\big)^{\frac{1}{2}}\exp \big(|\cdot|\frac{\eps^{2-\alpha}}2\int_{\R}V_0\big)$ in the weak coupling regime).

The numerical scheme used for computing the eigenfunctions is described in Appendix~\ref{sec.numerics}. We defined the oscillatory potential by
\[ q(X,y)=4\cos(2\pi y)\times \exp(-X^2/8)\,,\]
(see Figure~\ref{fig.q}), so that
\[ Q(X,y)=-\frac1{\pi^2}\cos(2\pi y)\times \exp(-X^2/8)\,, \qquad V_0(X)=-\frac2{\pi^2} \exp(-X^2/4)\,.\]
The value of the parameters are $\eps=1/5$, and $\alpha=2$ (semiclassical regime), $\alpha=1/2$ (weak coupling regime) and $\alpha=1$ (critical regime). 

\begin{figure}[!htb]
\subfigure[Semiclassical regime, $\alpha=2$, and $\eps=1/5$]{
\includegraphics[width=0.5\textwidth]{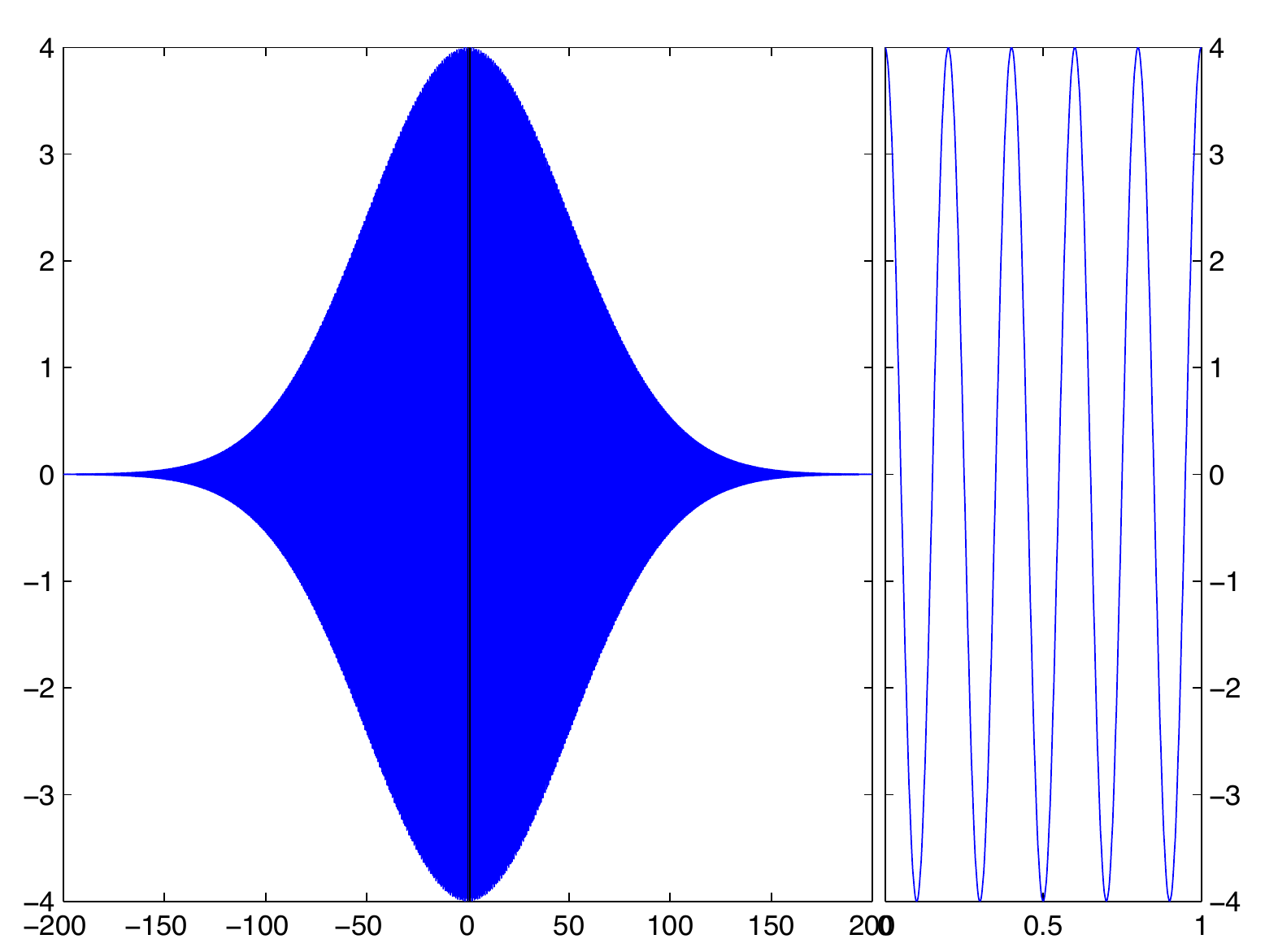}
}%
\subfigure[Weak coupling regime, $\alpha=1/2$, and $\eps=1/5$]{
\includegraphics[width=0.5\textwidth]{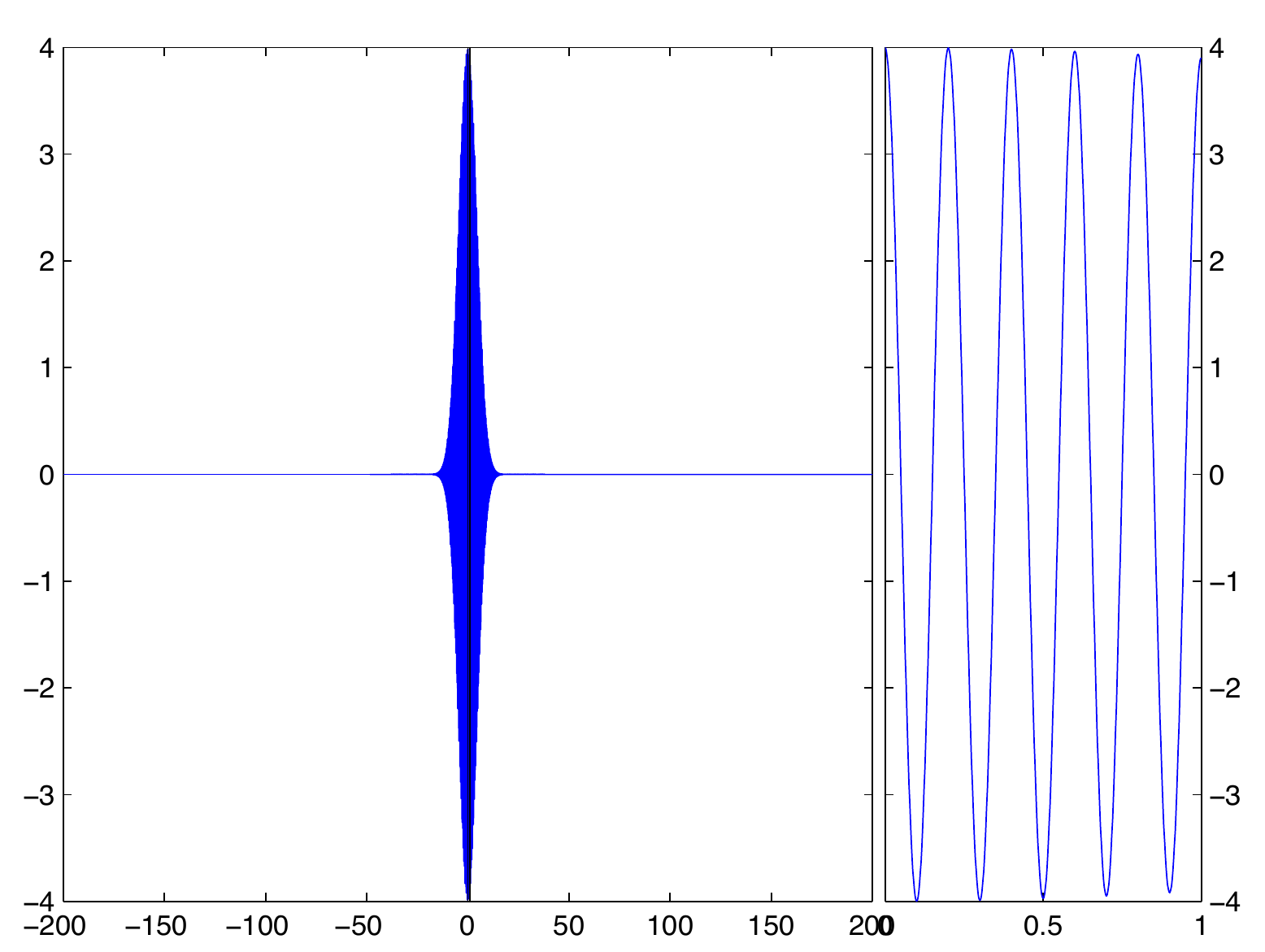}
}
\caption{Large and small-scale behavior of the oscillatory potential, $q(\eps^\alpha\cdot,\cdot/\eps)$.}
\label{fig.q}
\end{figure}%

\begin{figure}[!htb]\vspace{-.5cm}
\subfigure[First three normalized eigenfunctions, $\psi_{n,\eps,\alpha}$]{
\includegraphics[width=0.48\textwidth]{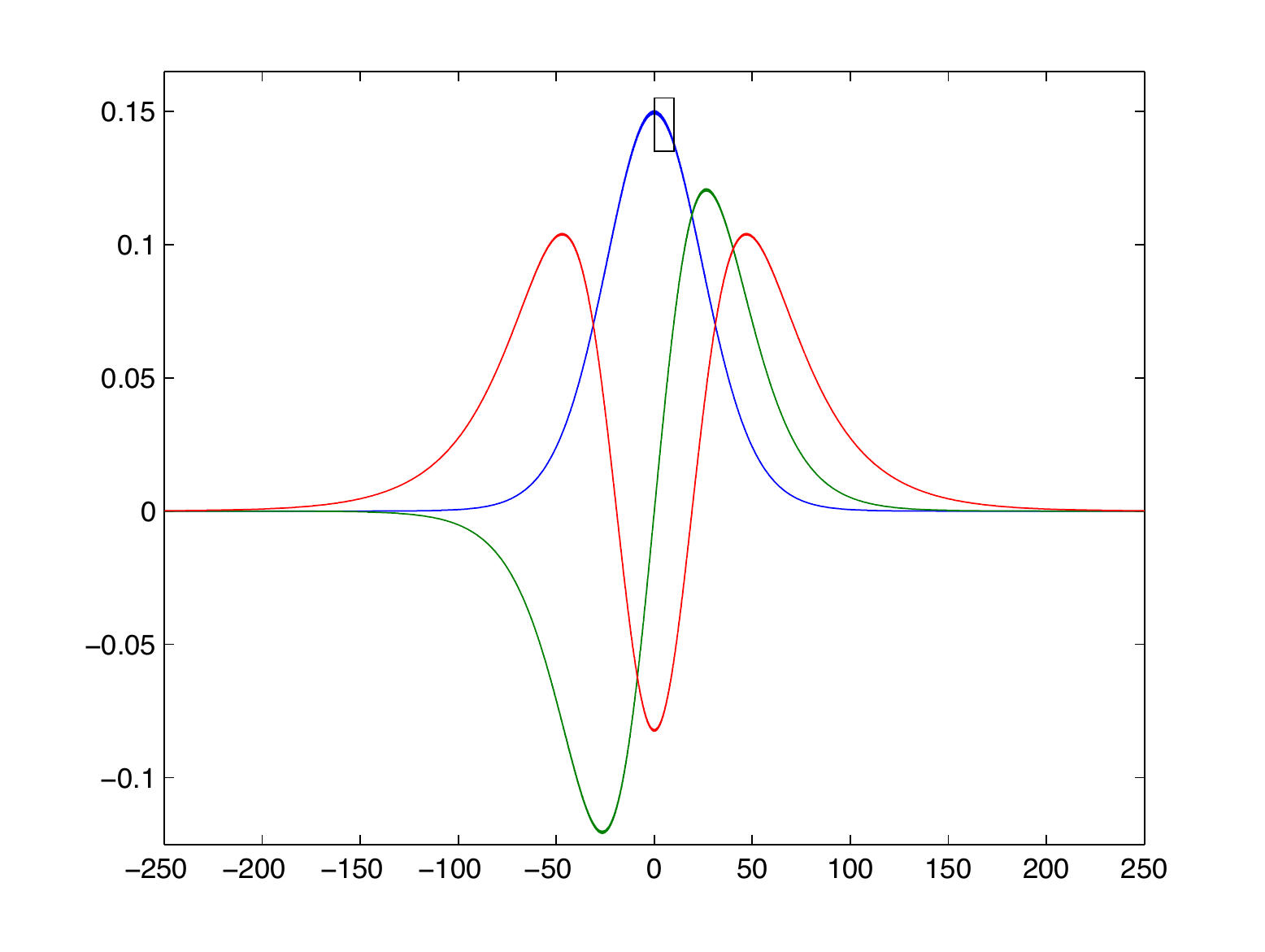}
}%
\subfigure[Close-up and comparison with effective eigenfunction, $\varphi^\eff_{1,\eps,\alpha}$, and its approximation, $\varphi^\app_{1,\eps,\alpha}$]{
\includegraphics[width=0.48\textwidth]{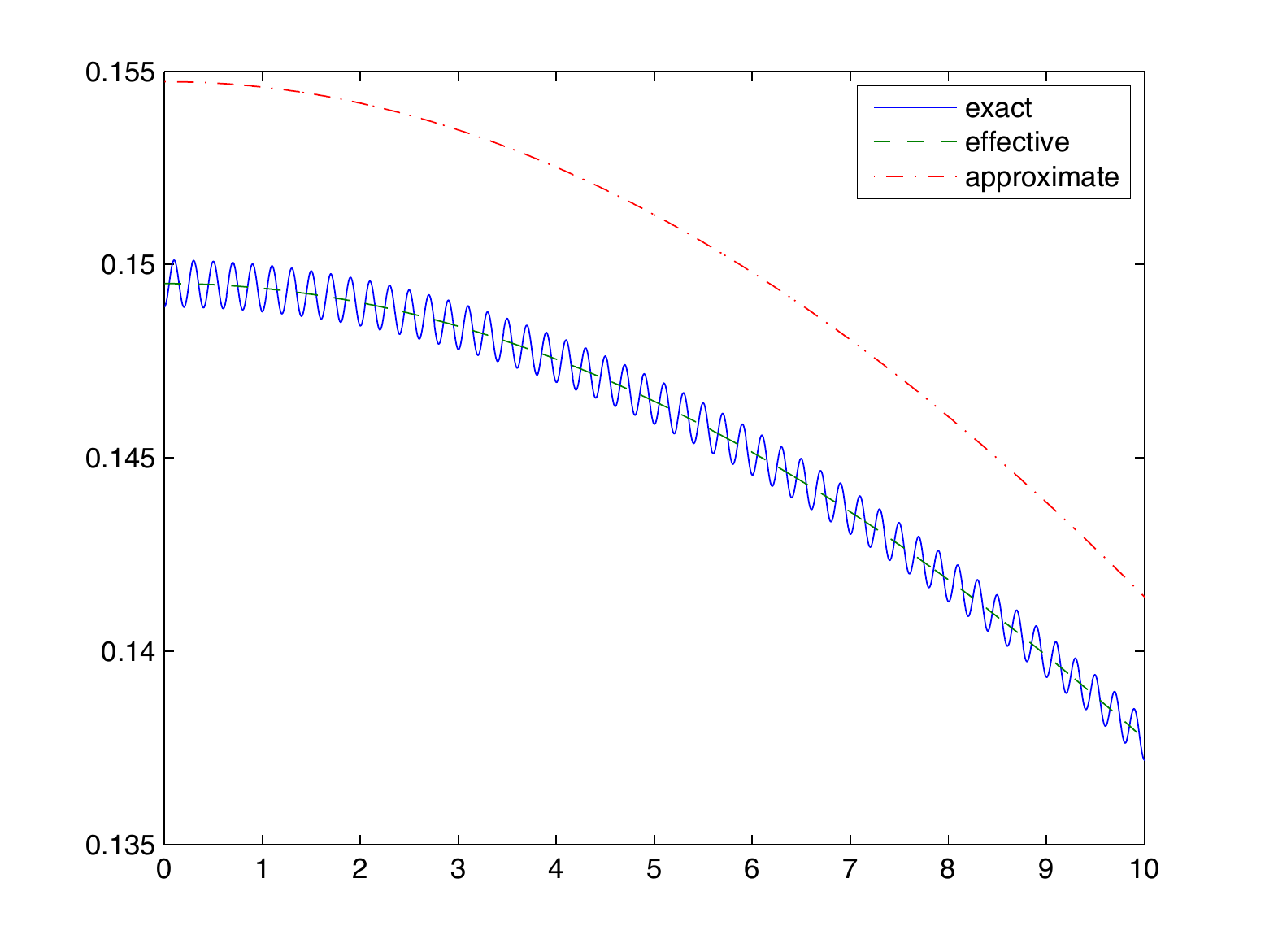}
}
\caption{Semiclassical regime, $\alpha=2$.}
\label{fig.alpha=2}
\end{figure}%
\begin{figure}[!htb]\vspace{-.5cm}
\subfigure[Normalized eigenfunction, $\psi_{\eps,\alpha}$]{
\includegraphics[width=0.48\textwidth]{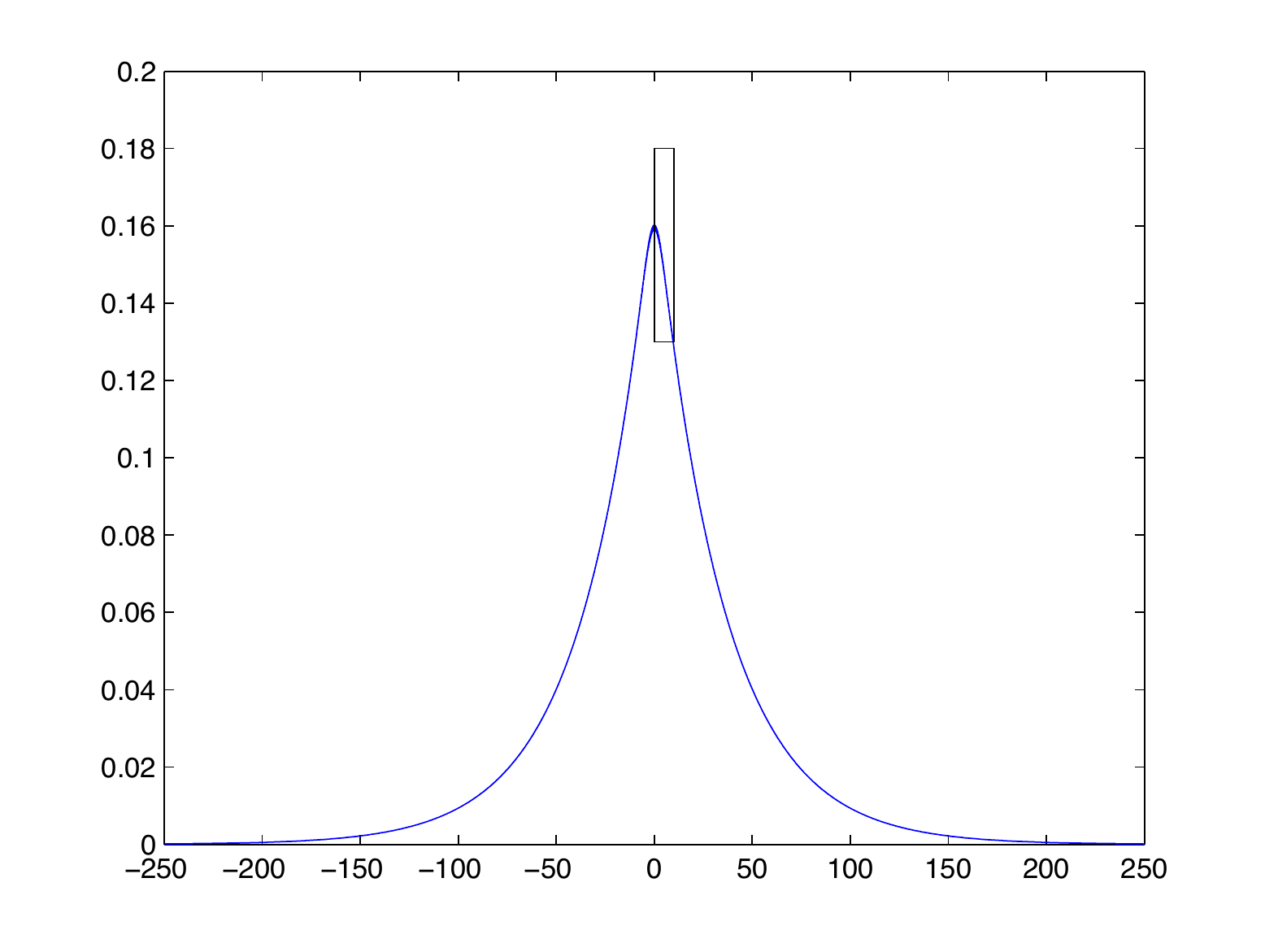}
}%
\subfigure[Close-up and comparison with effective eigenfunction, $\varphi^\eff_{\eps,\alpha}$, and its approximation, $\varphi^\app_{\eps,\alpha}$]{
\includegraphics[width=0.48\textwidth]{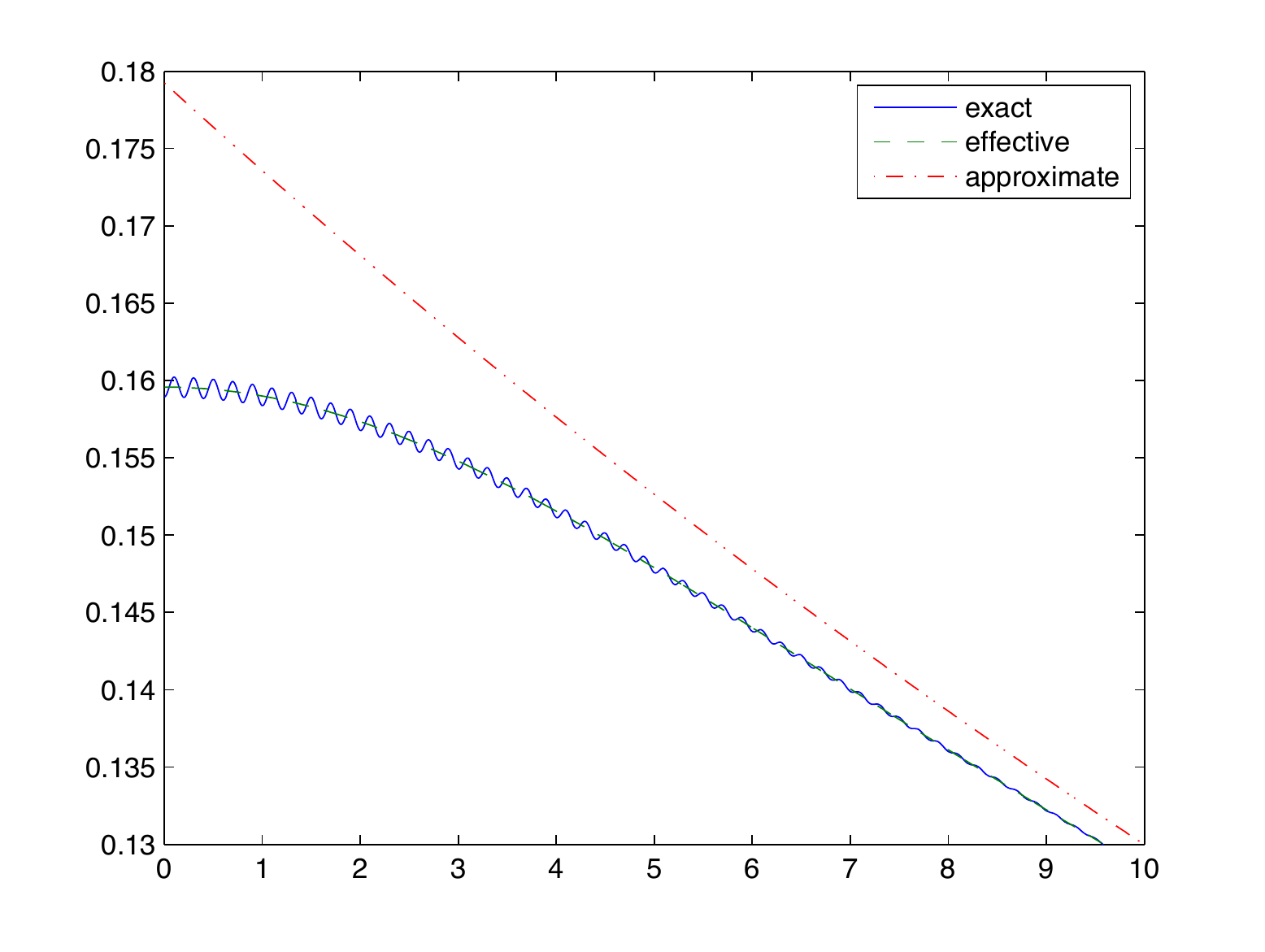}
}
\caption{Weak coupling regime, $\alpha=1/2$.}
\label{fig.alpha=05}
\end{figure}%
\begin{figure}[!htb]\vspace{-.5cm}
\subfigure[First normalized eigenfunction, $\psi_{1,\eps}$]{
\includegraphics[width=0.48\textwidth]{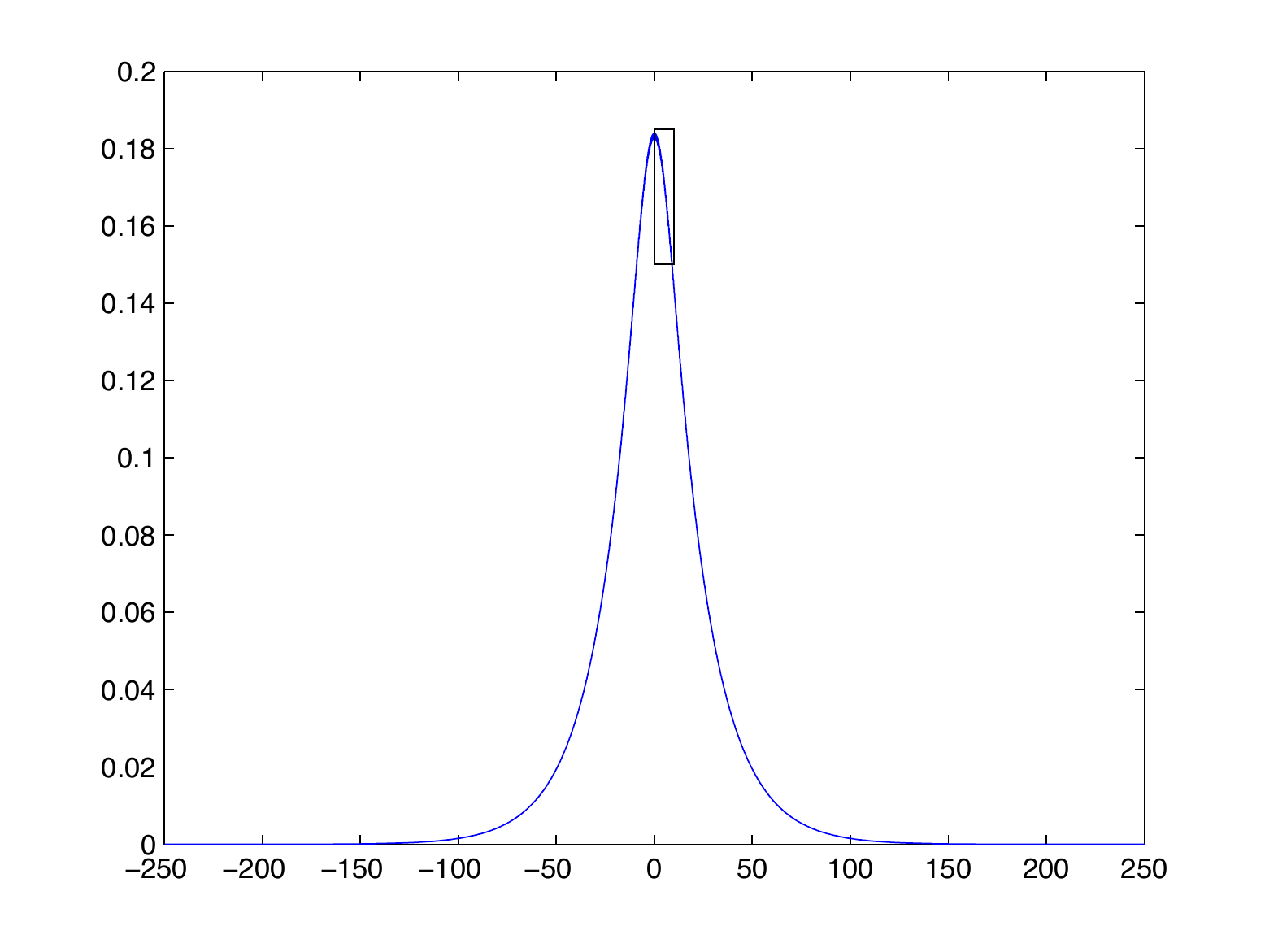}
}%
\subfigure[Close-up and comparison with effective eigenfunction, $\varphi^\eff_{1,\eps}=\eps\varphi_{1,V_0}(\eps^{1/2}\cdot)$]{
\includegraphics[width=0.48\textwidth]{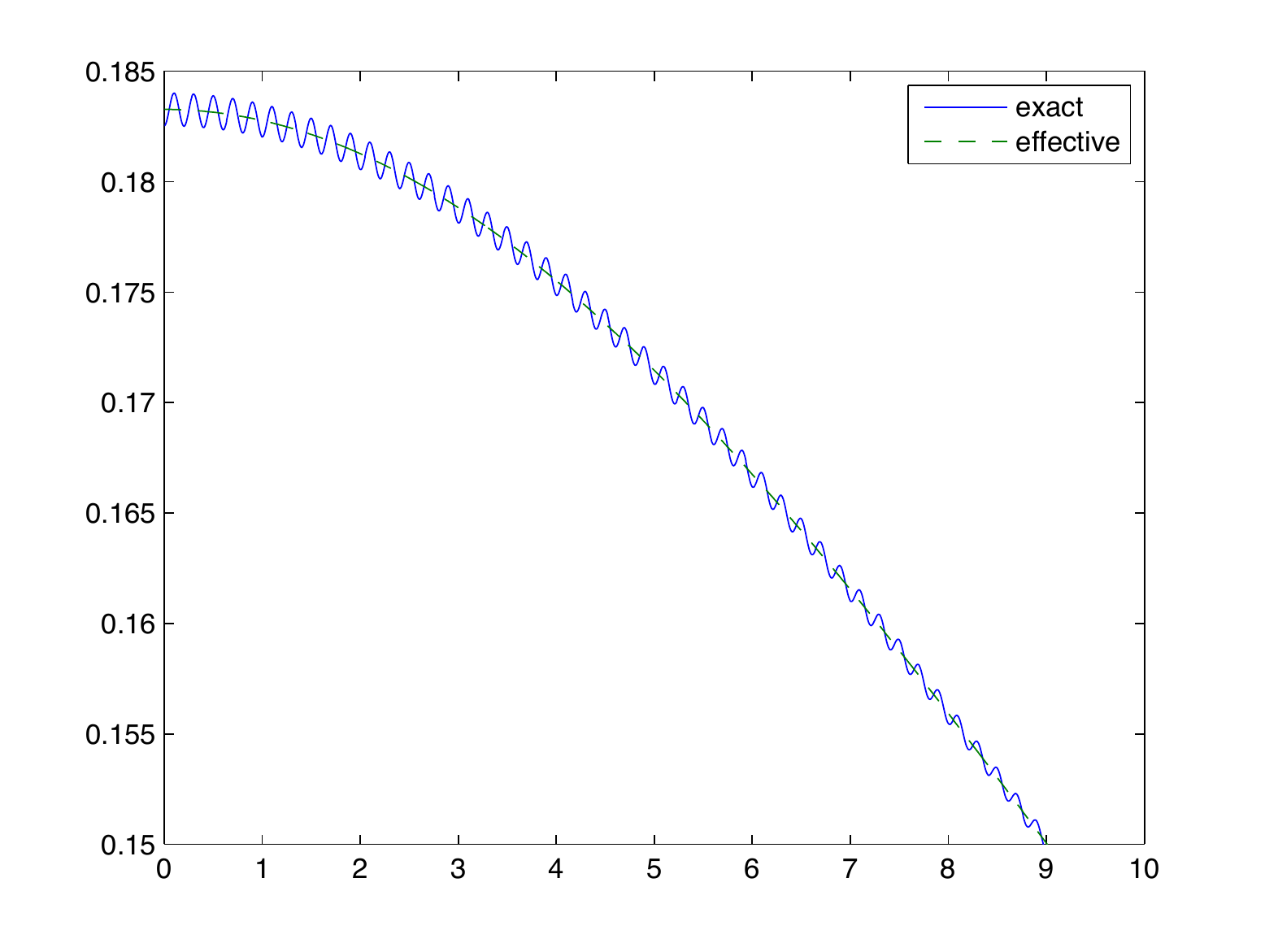}
}
\caption{Critical regime, $\alpha=1$.}
\label{fig.alpha=1}
\end{figure}

A striking observation on these numerical experiments is that, as suggested in Propositions~\ref{prop.semiclassical},~\ref{prop.small-amplitude} and~\ref{prop.critical}, the main source of imprecision arises when approximating the eigenfunction of the effective problem, $\varphi^\eff$ with its semiclassical or weak coupling asymptotics, $\varphi^\app$. Notice however that, in our results, the estimates between the exact and effective eigenfunctions is still much less precise than the eigenvalue approximation, due to the smallness of the spectral gap. As a matter of fact, the crudeness of the former estimates hides a finer structure for the eigenfunctions, suggested by our proof: we show that
\[\psi_{\eps,\alpha}(x) \approx e^{\phi_\eps (x)} \times \varphi^\eff_{\eps,\alpha}\left(\int_0^{x} e^{-2\phi_\eps(x')} \dx x' \right) \,, \quad \phi_\eps(x) \approx \eps^2 Q(\eps^\alpha x,x/\eps)\,.\]
Notice the above shows three different scales, as the scaling involved in $\varphi^\eff_{\eps,\alpha}$ is different form the ones involved by $\phi_\eps$, unless $\alpha=1$. In particular, oscillations are localized on a smaller region than the one defined by the eigenfunction in the weak coupling regime, in contrast with the semiclassical regime (which explains why a two-scale expansion could be obtained only in this situation). Although the precision of our results is insufficient to demonstrate the refined behavior, the latter is fully supported by our numerical simulations. Indeed, had we plotted
\[ \widetilde \varphi^\eff_{\eps,\alpha}(x)=e^{\eps^2 Q(\eps^\alpha x,x/\eps)}\times \varphi^\eff_{\eps,\alpha}(x)\,,\]
in Figures~\ref{fig.alpha=2},~\ref{fig.alpha=05} and~\ref{fig.alpha=1}, then its graph would have been superimposed with the one of the corresponding eigenfunction of the oscillatory operator, $\L$.

\subsection{Related results and analogies in the literature}
Let us now briefly discuss the relation of our results with the existing literature.

\subsubsection{Homogenization} The problem of describing the large scale behavior of partial differential equations with periodically oscillating coefficients on a small scale is often tackled by the so-called homogenization process. Our approach is closely related, as the effective potential can be thought as the effective medium obtained after solving the auxiliary ``cell'' problem~\eqref{eq.cell}. Notice however that our problem involves three scales whereas standard works based on homogenization techniques typically involve only two scales, as in the case $\alpha=0$ or $\alpha=1$. This transpires for instance in~\cite{AllaireMalige97,AllairePiatnitski02,ACPSV04} where, in order to deal with large potentials, the authors introduce a {\em factorization principle} which is similar to our normal form transformation (see below), although without the change of variable. 
We would also like to mention the recent works \cite{PankratovaPettersson15,ChechkinaPankratovaPettersson15,Pettersson}, where similar multiscale problem as ours are studied.

\subsubsection{Effective potential} The notion of effective problems is of course ubiquitous in asymptotic studies, and in particular when a microstructure is involved. As mentioned earlier, in the specific situation $\alpha=0$, the effective potential $V_0$ in~\eqref{Pb.effective} was introduced in the previous work~\cite{DucheneVukicevicWeinstein14}. This was then refined and generalized so as to deal with higher dimensions and scattering resonances in~\cite{Drouot15,Drouot16}. In particular, the correction $V_1$ in~\eqref{Pb.effective} was introduced in~\cite{Drouot15}, together with the existence of a power series expansion ---with the first terms being explicitly given through the effective potentials--- again in the situation $\alpha=0$.

\subsubsection{Anderson localization} It is interesting to notice that our results display a similar behavior than the one at stake in the famous Anderson localization phenomenon (see for instance \cite{FGK07}). In the present context, a discrete spectral structure emerges from the strong oscillations of the electric potential. Here, in some sense, \emph{the oscillations play the role of randomness}. The phenomenon is weakened by the spatial decay of the potential (in particular the essential spectrum is unperturbed), however more and more discrete eigenvalues are generated as $\eps$ is small and $\alpha$ is large.

\subsubsection{Born-Oppenheimer reduction}
Our strategy is somewhat reminiscent of the so-called Born-Oppenheimer approximation. This method of dimensional reduction, that is a quantum averaging strategy, was initially introduced in \cite{BO27} (see also \cite{KMSW92} in relation with molecular physics) and has been developed in many different contexts (see for instance the review \cite{J14}), and in particular to derive spectral asymptotic results (see \cite{LT13, Ray16}). Let us explain the analogy. Consider the two-dimensional operator
\[D^2_{x_{1}}+D^2_{x_{2}}+q(\eps^{\alpha}x_{1},\eps^{-1}x_{2})\,,\]
acting on $L^2(\R\times\S)$ with $q$ satisfying the same assumptions as above. This operator is formally obtained by introducing a fictitious variable $x_{2}$ and duplicating the second derivative. Using the rescaling $y=\eps^{-1}x_{2}$, we are reduced to the unitarily equivalent operator:
\[D^2_{x_{1}}+\eps^{-2}D^2_{y}+q(\eps^{\alpha}x_{1},y)\,.\]
This operator is in the Born-Oppenheimer form: it is partially semiclassical with respect to the variable $x_{1}$ and $\mathcal{M}_{\eps^\alpha x_{1},\eps}=\eps^{-2}D^2_{y}+q(\eps^{\alpha}x_{1},y)$ can be interpreted as an operator-valued potential. It can be proved in a rather general framework (see for instance \cite{Martinez07} in the context of pseudo-differential operators) that, generically, the low-lying spectrum is well described by the one of the reduced operator
\[D^2_{x_{1}}+\mu_{\eps}(\eps^{\alpha}x_{1})\approx D^2_{x_{1}}+\eps^2 V_0(\eps^{\alpha}x_{1})\,.\]
The study here is complicated by the presence of cross-derivatives $D_{x_1}D_{x_2}$, due to the fact that the variables $x_{1}$ and $x_{2}$ are not independent. The above strategy however served as a guideline for the construction of the normal form in a previous version of this work~\cite{DR16}.

\subsubsection{Weakly decaying oscillating potentials} The strategy of our paper follows some ideas of~\cite{Raikov16}, where the author studies the asymptotic distribution of discrete eigenvalues of the Schrödinger operator with weakly decaying (non-rapidly) oscillating potentials. It is well-known that if the potential is negative at infinity, then the asymptotic distribution of discrete eigenvalues is driven by the decay rate of the potential~\cite{RS4}, but no results were known when the potential oscillates at infinity. Here again, the effective potential plays a role as it allows to predict the correct behavior of the asymptotic distribution of discrete eigenvalues.

\subsection{Strategy and outline} 
Let us describe the key elements of the proof of our results. 
The main idea is a {\em normal form} transformation. We show that the contributions of the oscillatory potential $q(\eps^\alpha x,x/\eps)$ may be approximately factorized thanks to a two-scale function $\phi_\eps(x)=\eps^2\Phi_\eps(\eps^\alpha x,x/\eps)$. This function plays the role of a complex phase, and is chosen after considering the operator $e^{-\phi_\eps}\L e^{\phi_\eps}$ and so that
\[q(\eps^\alpha x,x/\eps)-\phi_{\eps,\alpha}''(x) - (\phi_{\eps,\alpha}')^2(x)=V_\eps(\eps^\alpha x)+r_\eps(x),\]
where $V_\eps$ depends only on the large scale and $r_\eps$ is a small remainder. We explain in Section~\ref{sec.phase} how $\phi_\eps$ may be explicitly constructed and how the effective potential $V_\eps=\eps^2 V_0+\eps^{3+\alpha} V_1$ emerges from this construction.
The normal form of the operator is made explicit and compared with the effective operator, $\L^\eff$, in Section~\ref{sec.conjugation}.

In Section \ref{sec.eigenvalues}, we apply the normal form transformation so as to compare the Rayleigh quotients associated with $\L$ and the ones associated with $\L^\eff$. Theorem~\ref{th.main-result-eigenvalues} quickly follows from the min-max principle. We then deduce the asymptotic behavior of the eigenvalues in the three aforementioned regimes.

In Section \ref{sec.eigenfunctions}, we use again the normal form of our operator to construct quasimodes. We detail in Section~\ref{sec.semiclassical} (respectively Section~\ref{sec.small-amplitude} and Section~\ref{sec.critical}) 
the semiclassical regime, $\alpha>1$ (respectively weak coupling regime, $\alpha<1$ and critical regime, $\alpha=1$). By using the classical resolvent bound for self-adjoint operators (consequence of the spectral theorem), together with the previously obtained results on the eigenvalues, allows to deduce the asymptotic behavior of the corresponding eigenfunctions, as stated above.

Section \ref{sec.WKB} is devoted to the case $\alpha=2$ and to the proof of the two-scale WKB expansion stated in Proposition \ref{prop.WKB}.

\section{The normal form}\label{sec.normal}

As mentioned earlier, the key ingredient of our analysis is a ``normal form'' of our operator, obtained through conjugation. The conjugation allows to replace the main oscillatory contributions by a large-scale, effective potential. More precisely, we shall introduce the transformation
\[ \T: \varphi\mapsto \psi, \quad \psi(x):=e^{\phi_{\eps,\alpha}(x)} \varphi\left(\int_0^{x} e^{-2\phi_{\eps,\alpha}(x')} \dx x' \right),\]
where $\phi_{\eps,\alpha}$ is to be determined. This yields a new Schrödinger operator involving the reduced potential (up to a complex phase which will not play a role in our analysis)
\[ V_{\phi_{\eps,\alpha}}(x):=q(\eps^\alpha x,x/\eps)-\phi_{\eps,\alpha}''(x) - (\phi_{\eps,\alpha}')^2(x).\]
Our aim is to choose $\phi_{\eps,\alpha}$ such that $V_{\phi_{\eps,\alpha}}$ depends only on the large-scale variable (up to a small oscillatory term).  We explain in section~\ref{sec.phase} how we obtain such a phase and how the effective potential surfaces. We then apply these results to our Schrödinger operator, and provide useful estimates in Section~\ref{sec.conjugation}

\subsection{Construction of the phase}\label{sec.phase}

Our aim in this section is to exhibit a complex phase, $\phi_{\eps,\alpha}(x)$, such that
\begin{equation} \label{eq.phase-equation}
q(\eps^\alpha x,x/\eps)-\phi_{\eps,\alpha}''(x) - (\phi_{\eps,\alpha}')^2(x)=V_\eps(\eps^\alpha x)+r_\eps,
\end{equation}
where the remainder term $r_\eps$ is as small as desired.
We seek $\phi_{\eps,\alpha}$ (and consequently $r_\eps$) with two-scale behavior
\[\phi_{\eps,\alpha}(x)=\eps^2\Phi_\eps(\eps^\alpha x ,x/\eps),\]
where the profile $\Phi_{\eps}$ has to be determined. Equation \eqref{eq.phase-equation} becomes
\begin{multline*}  \left(D_y^2+2\eps^{1+\alpha}D_XD_y+\eps^{2(1+\alpha)} D_X^2\right)\Phi_\eps(X,y)\\
+\eps^{2}\left( D_y\Phi_\eps+\eps^{1+\alpha} D_X\Phi_\eps\right)^2(X,y)+q(X,y)=V_\eps(X)+r_\eps(X,y).\end{multline*}
Consistently, we will solve the above with the following Ansatz:
\[\Phi_\eps(X,y)=\sum_{\kappa\in\NN^2}\eps^{p(\kappa)}\Phi_\kappa(X,y), \quad V_\eps(X)=\sum_{\kappa\in\NN^2}\eps^{p(\kappa)}\Phi_\kappa(X,y)\]
where $ p(\kappa_1,\kappa_2):=(1+\alpha)\kappa_1+2\kappa_2$, and we additionally enforce the property
\[\forall \kappa\in\NN^2, \quad \int_\S \Phi_\kappa(X,y)\dx y=0.\]
\subsubsection{Inductive construction}
Plugging the Ansatz into the equation and solving at order $(\kappa_1,\kappa_2)$ yields
\begin{multline*}
D_y^2\Phi_{\kappa_1,\kappa_2}+2D_XD_y\Phi_{\kappa_1-1,\kappa_2}+D_X^2\Phi_{\kappa_1-2,\kappa_2}\\
+ \sum_{\kappa_3=0}^{\kappa_1} (D_y\Phi_{\kappa_3,\kappa_2-1})(D_y\Phi_{\kappa_1-\kappa_3,\kappa_2-1})
+\sum_{\kappa_3=1}^{\kappa_1} (D_y\Phi_{\kappa_3-1,\kappa_2-1})(D_X\Phi_{\kappa_1-\kappa_3,\kappa_2-1})\\
+\sum_{\kappa_3=2}^{\kappa_1} (D_X\Phi_{\kappa_3-2,\kappa_2-1})(D_X\Phi_{\kappa_1-\kappa_3,\kappa_2-1})=-\delta_{(\kappa_1,\kappa_2)=(0,0)}q+V_{\kappa_1,\kappa_2},
\end{multline*}
where $\delta_{(\kappa_1,\kappa_2)=(0,0)}=1$ if $(\kappa_1,\kappa_2)=(0,0)$ and $0$ otherwise.

We may define $\Phi_{\kappa_1,\kappa_2}$ and $V_{\kappa_1,\kappa_2}$ by induction on $|\kappa|=\kappa_1+\kappa_2$. Indeed, the equation reads
\[ D_y^2\Phi_{\kappa_1,\kappa_2}(X,y)+F_{\kappa_1,\kappa_2}(X,y)=V_{\kappa_1,\kappa_2}(X),\]
where $F_{\kappa_1,\kappa_2}(X,y)$ is given explicitly in terms of $\Phi_{\kappa'}$ with $|\kappa'|\leq |\kappa-1|$. We solve the equation by setting (as forced by the Fredholm alternative in the variable $y$)
\[ V_{\kappa_1,\kappa_2}(X)=\int_\S F_{\kappa_1,\kappa_2}(X,y)\dx y,\]
and then $\Phi_{\kappa_1,\kappa_2}$ is the unique mean-zero solution to
\[D_y^2\Phi_{\kappa_1,\kappa_2}(X,y)=V_{\kappa_1,\kappa_2}(X)-F_{\kappa_1,\kappa_2}(X,y).\]

Let us detail the first terms. At order $\kappa=(0,0)$, one finds
\[D_y^2\Phi_{0,0}(X,y)+q(X,y)=V_{0,0}(X).\]
Because , $y\mapsto q(X,y)$ is mean-zero, we deduce
\[V_{0,0}(X)=0 \quad \text{ and } \quad \Phi_{0,0}(X,y)=Q(X,y),\]
where we recall that $Q(X,\cdot)$ is the zero-mean solution to $
D_{y}^2Q(X,y)=-q(X,y)$.
At order $\kappa=(1,0)$, one has
\[D_y^2\Phi_{1,0}(X,y)+2D_XD_y\Phi_{0,0}(X,y)=V_{1,0}(X).\]
We deduce
\[V_{1,0}(X)=0 \quad \text{ and } \quad D_y^2\Phi_{1,0}(X,y)=-2D_XD_yQ(X,y).\]
At order $(k,0)$, for $k\geq 2$, one has
\[D_y^2\Phi_{k,0}(X,y)+2D_XD_y\Phi_{k-1,0}(X,y)+D_X^2\Phi_{k-2,0}(X,y)=V_{k,0}(X).\]
We deduce
\[V_{k,0}(X)=0 \quad \text{ and } \quad D_y^{2}\Phi_{k,0}=-2D_XD_y\Phi_{k-1,0}-D_X^2\Phi_{k-2,0}.\]
In particular, 
\[ D_y^{2}\Phi_{2,0}=3D_X^2Q \quad \text{ and } \quad D_y^{2}\Phi_{2,0}=2D_X^2\Phi_{1,0}.\]
At order $\kappa=(0,1)$, we find
\[D_y^2\Phi_{0,1}(X,y)+(D_y\Phi_{0,0})^2(X,y)=V_{0,1}(X).\]
We deduce from the Fredholm alternative
\[V_{0,1}(X)=\int_\S (D_yQ)^2\dx y=:V_{0}(X), \]
where $V_0$ was already defined in~\eqref{eq.V}, and
\[D_y^2\Phi_{0,1}=-(D_yQ)^2(X,y)-V_{0}(X).\]
At order $\kappa=(1,1)$, one has
\[D_y^2\Phi_{1,1}+2(D_y\Phi_{1,0})(D_y\Phi_{0,0})+ 2(D_X\Phi_{0,0})(D_y\Phi_{0,0})+2D_XD_y\Phi_{0,1}=V_{1,1}.\]
This yields, as above,
\[V_{1,1}(X)=-2\int_\S \big((D_yQ) (D_XQ)\big)(X,y)\dx y=:V_1(X) .\]
where  $V_1$ was also defined in~\eqref{eq.V}, and
\[D_y^2\Phi_{1,1}(X,y)=V_{1,1}(X)+2\big((D_yQ) (D_XQ)\big)(X,y)-2D_XD_y\big((D_yQ)^2\big)(X,y).\]

\subsubsection{Truncated expansion}

As it is clear from the previous section, we can solve the problem up to any arbitrary high order provided that ${q\in \bigcap_{\ell\in \NN} W^{\ell,\infty}(\R\times\S)}$. Unfortunately, other sources of error cannot be avoided, and our results do not benefit from adding all high-order contributions in the expansion. Here we decide to truncate the expansion so as to satisfy~\eqref{eq.phase-equation} with precision $\O(\eps^{\min\{4,2+2(1+\alpha),4(1+\alpha)\}})$.
\begin{definition}\label{def.phi}
For $\eps>0$ and $q\in W^{3,\infty}(\R\times\S)$, we set
\[\phi_{\eps,\alpha}(x)=\eps^2\left(\sum_{k=0}^3\eps^{k(1+\alpha)}\Phi_{k,0}+\eps^2 \Phi_{0,1}+\eps^{3+\alpha}\Phi_{1,1}\right)(\eps^\alpha x,x/\eps),\]
where $\Phi_{\kappa_1,\kappa_2}$ are defined as the unique mean-zero solutions to
\begin{align*}
&-D_y^2 \Phi_{0,0}=q , \qquad -D_y^2 \Phi_{1,0}=2D_XD_y \Phi_{0,0} \, \qquad -D_y^2 \Phi_{2,0}=-3D_X^2\Phi_{0,0} ,\\
& -D_y^2 \Phi_{3,0}=-2D_X^2\Phi_{1,0}  ,  \qquad  -D_y^2\Phi_{0,1}=(D_y\Phi_{0,0})^2-\int_\S (D_y\Phi_{0,0})^2\dx y,\\
& -D_y^2\Phi_{0,1}=-2(D_y\Phi_{0,0})(D_X\Phi_{0,0})+2\int_\S (D_y\Phi_{0,0})(D_X\Phi_{0,0})\dx y+2D_XD_y\Phi_{0,1}.\end{align*}
\end{definition}

\begin{lemma}\label{lem.phi} Assume $\alpha>-1$ and $q\in W^{5,\infty}(\R\times\S)$. Then one has
\[\Norm{e^{-\phi_{\eps,\alpha}}}_{W^{2,\infty}(\R)}+\Norm{e^{\phi_{\eps,\alpha}}}_{W^{2,\infty}(\R)}\leq  C(\Norm{q}_{W^{3,\infty}(\R\times\S)}) \,,\]
\[\Norm{e^{\phi_{\eps,\alpha}}-1}_{L^\infty(\R)}\leq \eps^2 C(\Norm{q}_{W^{3,\infty}(\R\times\S)}) \,,\] 
and
\[ \Norm{q(\eps^\alpha\cdot,\cdot/\eps) -\phi_{\eps,\alpha}''- (\phi_{\eps,\alpha}')^2-V_\eps(\eps^\alpha \cdot)}_{L^\infty(\R)} \leq \eps^{\min\{4,4(1+\alpha)\}} C(\Norm{q}_{W^{5,\infty}(\R\times\S)})\,,\]
where we define
\begin{align*}
V_\eps(X)&=-\eps^2\int_\S \abs{\partial_y Q}^2(X,y)\dx y\ +\  2 \eps^{3+\alpha}\int_\S \big((\partial_x Q)(\partial_y Q)\big)(X,y)\dx y\\
&=:\eps^2 V_0(X)+\eps^{3+\alpha} V_1(X)\,.
\end{align*}
Here, $Q$ is defined by~\eqref{eq.cell} and satisfies
\[\Norm{Q}_{W^{1,\infty}(\R\times\S)}+\Norm{D_y Q}_{L^\infty(\R\times\S)}\leq C(\Norm{q}_{W^{1,\infty}(\R\times\S)}).\]
\end{lemma}
\begin{proof} Let us start with the last estimate of the statement. We test~\eqref{eq.cell} with $Q$, integrate by parts and use Cauchy-Schwarz inequality to deduce 
\[\Norm{D_{y}Q(X,\cdot)}^2_{L^2(\S)}\leq \Norm{q}_{L^\infty(\R\times\S)}\Norm{Q(X,\cdot)}_{L^2(\S)}\leq \frac1{2\pi}\Norm{q}_{L^\infty(\R\times\S)}\Norm{D_yQ(X,\cdot)}_{L^2(\S)}\,,\]
where we used Wirtinger's inequality (or the min-max principle since the mean value of $Q$ is zero, and the second lowest eigenvalue of $D_y^2$ on $L^2(\S)$ is $4\pi^2$) for the last inequality. Moreover, we obviously have 
\[\Norm{D_{y}^2Q(X,\cdot)}^2_{L^2(\S)}\leq \Norm{q}_{L^\infty(\R\times\S)}\,,\]
and thus we control $\Norm{ Q}_{L^\infty(\R\times\S)}+\Norm{D_y Q}_{L^\infty(\R\times\S)}$. There remains to control $\Norm{D_X Q}_{L^\infty(\R\times\S)}$, which is obtained in the same way, after differentiating~\eqref{eq.cell} with respect to $X$.

By proceeding as above, we obtain
\[\Norm{\phi_{\eps,\alpha}}_{L^\infty(\R)}+\eps \Norm{\phi_{\eps,\alpha}'}_{L^\infty(\R)}+\eps^2 \Norm{\phi_{\eps,\alpha}''}_{L^\infty(\R)}\leq \eps^2 C(\Norm{q}_{W^{3,\infty}(\R\times\S)})\,,\]
and the first two estimates of the statement follow immediately. For the remaining one, we find
\begin{align*}\phi_{\eps,\alpha}''+ (\phi_{\eps,\alpha}')^2&=\eps^2 \sum_{\kappa}\eps^{p(\kappa)} \left(\eps^{-2}D_y^2+2\eps^{\alpha-1}D_XD_y+\eps^{2\alpha} D_X^2\right)\Phi_\kappa(X,y)\\
&\qquad +\eps^2\left(  \sum_{\kappa}  \eps^{-1}\eps^{p(\kappa)}  D_y\Phi_\kappa+\eps^{\alpha} \eps^{p(\kappa)} D_X  \Phi_\kappa\right)^2(X,y)\\
&=\sum_{0\leq |\kappa|\leq 8} \eps^{p(\kappa)}r_\kappa(X,y).
\end{align*}
By our construction of $\Phi_\kappa$, we ensured $r_{0,0}=-q_\eps$, $r_{1,0}=r_{2,0}=r_{3,0}=0$, as well as $r_{0,1}=V_0$ and $r_{1,1}=V_1$. The other terms satisfy
\[\Norm{r_\kappa}_{L^\infty(\R\times\S)}\leq C(\Norm{q}_{W^{5,\infty}(\R\times\S)})\]
and $p(\kappa)\geq \min\{p(4,0),p(2,1),p(0,2)\}=\min\{4(1+\alpha),2(1+\alpha)+2,4\}$. 
\end{proof}

\subsection{The conjugation}\label{sec.conjugation}

Thanks to Definition~\ref{def.phi} and Lemma~\ref{lem.phi}, we may now introduce the normal form of our operator, $\L$, through the following transformation. 
\begin{lemma} \label{lem.normal-form-transform}
Let $\alpha>-1$ and $q\in W^{3,\infty}(\R)$. The application
\[ \T: \varphi\mapsto \psi, \quad \psi(x):=e^{\phi_{\eps,\alpha}(x)}\varphi\left(\int_0^{x} e^{-2\phi_{\eps,\alpha}(x')} \dx x' \right)\]
defines a
continuous isomorphism from $H^k(\R)$ into $H^k(\R)$ for ${k=0,1,2}$, and one has
\begin{align*} \Norm{\T( \varphi)-\varphi}_{L^2(\R)}&\leq \eps^2 C(\Norm{q}_{W^{3,\infty}(\R\times\S)})   \Norm{ \varphi}_{L^2(\R)}\,,\\
\Norm{\T( \varphi)}_{H^k(\R)}\ + \ \Norm{\T^{-1}(\varphi)}_{H^k(\R)}&\leq C(\Norm{q}_{W^{3,\infty}(\R\times\S)}) \Norm{ \varphi}_{H^k(\R)}\,.
\end{align*}
\end{lemma}
\begin{remark}The change of variable $\tilde x:=\int_0^{x} e^{-2\phi_{\eps,\alpha}(x')} \dx x'$
is used to eliminate derivatives of order 1 in Lemma~\ref{lem.normal-form} below. It turns out that this change of variable is crucial for the construction of quasimodes (see Section \ref{sec.eigenfunctions}), but not so much for the estimate of the eigenvalues (Section \ref{sec.comparison}, Theorem \ref{th.comparison-eigenvalues}), where the simple factorization by $e^{\phi_{\eps,\alpha}}$ would suffice. 
\end{remark}
The normal form is obtained by considering $\T^{-1}\L \T$. More precisely, we will make use of the following identity.
\begin{lemma} \label{lem.normal-form}
Let $\alpha>-1$ and $\eps\in(0,\eps_0)$ as in Lemma~\ref{lem.phi}.
Let $\varphi\in H^2(\R)$ and $\psi(x):=\T(\varphi)(x)$. 
Then
\[(\L\psi)(x) =e^{-3\phi_{\eps,\alpha}(x)}\big(D_x^2+V^{\red}_{\eps,\alpha}(x)\big)\varphi(\tilde x)\,,\]
where we denote $\tilde x=\int_0^x e^{-2\phi_{\eps,\alpha}(x')}\dx x'$, and
\begin{equation}\label{eq.edf-Vred}
V^{\red}_{\eps,\alpha}(x):=e^{4\phi_{\eps,\alpha}(x)} \left( q(\eps^\alpha x,x/\eps) -\phi_{\eps,\alpha}''(x)-(\phi_{\eps,\alpha}')^2(x) \right) .
\end{equation}
\end{lemma}
   \begin{proof}
Since $V^{\red}_{\eps,\alpha},e^{-\phi_{\eps,\alpha}(x')}\in L^\infty(\R)$, and $\T$ defines a continuous isomorphism from $H^2(\R)$ into $H^2(\R)$, the following identities are well-defined in $L^2(\R)$. One has
\begin{align*} (\L\psi)(x) &:= \left( D_x^2+q(\eps^\alpha x,x/\eps)\right)\psi(x) \\
& =  q(\eps^\alpha x,x/\eps)e^{\phi_{\eps,\alpha}(x)}\varphi(\tilde x) -i D_x\left(\phi_{\eps,\alpha}'(x)e^{\phi_{\eps,\alpha}(x)}\varphi(\tilde x)+e^{-\phi_{\eps,\alpha}(x)}\varphi'(\tilde x)\right)\\
& = e^{\phi_{\eps,\alpha}(x)}\varphi(\tilde x) \left( q(\eps^\alpha x,x/\eps) -\phi_{\eps,\alpha}''(x)-(\phi_{\eps,\alpha}')^2(x) \right) +e^{-3\phi_{\eps,\alpha}(x)}(D_x^2\varphi)(\tilde x)\\
&=e^{-3\phi_{\eps,\alpha}(x)}\big(D_x^2+V^{\red}_{\eps,\alpha}(x)\big)\varphi(\tilde x)\,,
\end{align*}
which concludes the proof.
\end{proof}
It is now natural to compare the normal form of our operator with the effective operator, $\L^\eff :=D_x^2+\eps^2 V_0(\eps^\alpha x)+\eps^{3+\alpha}V_1(\eps^\alpha x)$. Indeed, by construction of $\phi_{\eps,\alpha}$, one has the following approximation.
\begin{lemma} \label{lem.normal-form-vs-effective}
Let $\alpha>-1$ and $q\in W^{5,\infty}(\R\times\S)$ be such that $\langle \cdot\rangle V_0',\langle \cdot\rangle V_1'\in L^\infty(\R)$, where we recall  \[V_0(X)=-\int_{\S}|\partial_{y}Q(X, y)|^2\dx y\quad \text{ and } \quad V_1(X)=2 \int_\S \big((\partial_x Q)(\partial_y Q)\big)(X,y)\dx y\] with $Q(X,\cdot)\in L^2(\S)$ the zero-mean solution to $
D_{y}^2Q(X,y)=-q(X,y)$. One has
\[\Norm{V^{\red}_{\eps,\alpha}(x)-\eps^2V_0\left(\eps^\alpha  \tilde x\right)-\eps^{3+\alpha}V_1\left(\eps^\alpha  \tilde x\right)}_{L^\infty(\R)}\leq C\,\eps^{\min\{4,4(1+\alpha)\}}\,,\]
where $C=C(\Norm{q}_{W^{5,\infty}(\R\times\S)},\Norm{\langle \cdot\rangle V_0'}_{L^\infty} ,\Norm{\langle \cdot\rangle V_1'}_{L^\infty} )$ and $\tilde x=\int_0^x e^{-2\phi_{\eps,\alpha}(x')}\dx x'$.
   \end{lemma}
   \begin{proof}
By Lemma~\ref{lem.phi}, one has 
\[\Norm{e^{4\phi_{\eps,\alpha}}-1}_{L^\infty(\R)}\leq \eps^2 C(\Norm{q}_{W^{3,\infty}(\R\times\S)}) \,,\]
and
\[ \Norm{q(\eps^\alpha\cdot,\cdot/\eps) -\phi_{\eps,\alpha}''- (\phi_{\eps,\alpha}')^2-V_\eps(\eps^\alpha \cdot)}_{L^\infty(\R)} \leq \eps^{\min\{4,4(1+\alpha)\}} C(\Norm{q}_{W^{5,\infty}(\R\times\S)})\,,\]
and
\[ \Norm{V_\eps}_{L^\infty(\R)}\leq \eps^{2} C(\Norm{q}_{W^{3,\infty}(\R\times\S)}) \,.\]
Collecting the above information and triangular inequalities yield
\[\Norm{V^{\red}_{\eps,\alpha}(x)-\eps^2V_0\left(\eps^\alpha  x\right)-\eps^{3+\alpha}V_1\left(\eps^\alpha  x\right)}_{L^\infty(\R)}\leq \eps^{\min\{4,4(1+\alpha)\}} C(\Norm{q}_{W^{5,\infty}(\R\times\S)})\,.\]

Now, we remark that
\[\abs{\tilde x-x}=\left|\int_0^{x} e^{-2\phi_{\eps,\alpha}(x')} \left(e^{2\phi_{\eps,\alpha}(x')}-1\right) \dx x'\right|\leq C(\Norm{q}_{W^{3,\infty}(\R\times\S)}) \eps^2 |x|\,.\]
By the Taylor formula, one has, for $V=V_0$ or $V=V_1$ and $x\neq 0$,
\[
V(\eps^\alpha\tilde x)-V(\eps^\alpha x)\\
=\frac{\tilde x-x}{x}\int_{0}^1r_{\eps}(x,t)^{-1} (\eps^\alpha x r_{\eps}(x,t))V'\left(\eps^\alpha x r_{\eps}(x,t)\right)\dx t\,,
\]
where
\[r_{\eps}(x,t):=1+t\frac{\tilde x-x}{x}\,.\]
We deduce that
\[ \Norm{ V(\eps^\alpha\tilde x)-V(\eps^\alpha x)}_{L^\infty(\R)}\ \leq C(\Norm{q}_{W^{3,\infty}(\R\times\S)}) \Norm{\langle \cdot\rangle V'}_{L^\infty}  \eps^{2}\,,\]
and the desired estimate follows by triangular inequality.
\end{proof}

\section{Asymptotic analysis of the eigenvalues}\label{sec.eigenvalues}

\subsection{Comparison of eigenvalues}\label{sec.comparison}
In this section, we prove Theorem~\ref{th.main-result-eigenvalues}, recalled below.

\begin{theorem} \label{th.comparison-eigenvalues}
Let $\alpha>-1$ and $q\in W^{5,\infty}(\R\times \S)$ be such that $\langle \cdot\rangle V_0',\langle \cdot\rangle V_1'\in L^\infty$.  There exists $C>0$ such that for all $n\in\NN$ and $\eps\in(0,1]$, one has
 \[  \abs{\lambda_{n,\eps,\alpha} - \lambda_{n,\eps,\alpha}^{\eff} } \leq C \eps^{\min\{4,4(1+\alpha)\}}\,.\]
\end{theorem}
\begin{proof}
The result is based on the min-max principle; thus we introduce the quadratic forms associated with our operators, respectively
\[ \Q(\psi):=\int_{\R}  \abs{\psi'}^2(x)+q(\eps^\alpha x,x/\eps)\abs{\psi}^2(x)\dx x\,,\]
and
\[ \Q^{\eff}( \psi)=\int_{\R}   \abs{ \psi'}^2( x)+\left(\eps^2V(\eps^\alpha x)+\eps^{3+\alpha}V(\eps^\alpha x)\right)\abs{ \psi}^2( x) \dx  x\,.\]

Let $f\in H^1(\R)$. Then $e^{\phi_{\eps,\alpha} }f\in H^1(\R)$ by Lemma~\ref{lem.phi}, and one has
\begin{align*}
\Q(e^{\phi_{\eps,\alpha}} f)&=\int_{\R}  \left(\abs{e^{\phi_{\eps,\alpha}} f'}^2+\abs{\phi_{\eps,\alpha}'e^{\phi_{\eps,\alpha}} f}^2+2e^{2\phi_{\eps,\alpha}}\phi_{\eps,\alpha}'ff'\right)(x) +q(\eps^\alpha x,x/\eps)\abs{e^{\phi_{\eps,\alpha}} f}^2(x)\dx x\\
&=\int_{\R}  \abs{e^{\phi_{\eps,\alpha}} f'}^2+\abs{ f}^2\left(\big(\phi_{\eps,\alpha}'e^{\phi_{\eps,\alpha}}\big)^2(x)-\big(2e^{2\phi_{\eps,\alpha}}\phi_{\eps,\alpha}'\big)'(x) +q(\eps^\alpha x,x/\eps)e^{2\phi_{\eps,\alpha}} \right)\dx x\\
&=\int_{\R}  e^{2\phi_{\eps,\alpha}(x)} \left(\abs{f'}^2(x)+\abs{f}^2(x)\left(-(\phi_{\eps,\alpha}')^2(x)-\phi_{\eps,\alpha}''(x) +q(\eps^\alpha x,x/\eps)\right)\right)\dx x\\
&=\int_{\R}  e^{2\phi_{\eps,\alpha}(x)} \left(\abs{f'}^2(x)+\abs{f}^2(x) e^{-4\phi_{\eps,\alpha}(x)} V^{\red}_{\eps,\alpha}(x) \right)\dx x\,.
\end{align*}
Now, we apply the near-identity change of variable
\[ \tilde x=\int_0^{x} e^{-2\phi_{\eps,\alpha}(x')} \dx x' \quad \Longleftrightarrow \quad x=\theta(\tilde x)\,,\]
and deduce
\[ \Q(e^{\phi_{\eps,\alpha}} f)=\int_{\R}  e^{4\phi_{\eps,\alpha}(\theta(\tilde x))} \left(\abs{f'}^2(\theta(\tilde x))+\abs{f}^2(\theta(\tilde x))e^{-4\phi_{\eps,\alpha}(\theta(x)) }V^{\red}_{\eps,\alpha}(\theta(\tilde x)) \right)\dx \tilde x\,.\]
Finally, denoting $\psi(\tilde x)=f(\theta(\tilde x))$, one has
\[ \Q(e^{\phi_{\eps,\alpha}} f)=\int_{\R}   \abs{\psi'}^2(\tilde x)+V^{\red}_{\eps,\alpha}(\theta(x))\abs{\psi}^2(\tilde x) \dx \tilde x\,,\]
that is to say
\[ \Q(\T(\psi))=\Q^{\eff}(\psi)+r(\psi)\,,\]
with
\[r(\psi):=\int_{\R}  \abs{\psi}^2(\tilde x) \left( V^{\red}_{\eps,\alpha}(\theta(\tilde x))  -\eps^2 V_0(\eps^\alpha \tilde x)-\eps^{3+\alpha} V_0(\eps^\alpha \tilde x)\right)\dx x\,.\]
By Lemma~\ref{lem.normal-form-vs-effective}, one has
\[\abs{r(\psi)}\leq \eps^{\min\{4,4(1+\alpha)\}} C(\Norm{q}_{W^{5,\infty}(\R\times\S)},\Norm{\langle \cdot\rangle V_0'}_{L^\infty},\Norm{\langle \cdot\rangle V_1'}_{L^\infty} )\Norm{\psi}_{L^2(\R)}^2.\]

By the min-max principle, and using Lemma~\ref{lem.normal-form-transform}, we deduce that
\[\lambda_{n,\eps,\alpha} \leq \lambda_{n,\eps,\alpha}^{\eff} + C \eps^2 \abs{\lambda_{n,\eps,\alpha}^{\eff}} + C \eps^{\min\{4,4(1+\alpha)\}}\]
and
\[\lambda_{n,\eps,\alpha}^{\eff} \leq \lambda_{n,\eps,\alpha}+C\eps^2 \abs{\lambda_{n,\eps,\alpha}} + C \eps^{\min\{4,4(1+\alpha)\}},\]
with $C=C(\Norm{q}_{W^{4,\infty}(\R\times\S)},\Norm{\langle \cdot\rangle V'}_{L^\infty} )$.
By Lemma~\ref{lem.phi}, we have
\[
\forall X\in\R, \quad  V_0(X)\leq  C(\Norm{q}_{L^{\infty}(\R\times\S)}) \quad \text{ and } \quad V_1(X)\leq  C(\Norm{q}_{W^{1,\infty}(\R\times\S)})\,.
\]
We deduce $0\leq -\lambda_{n,\eps,\alpha}^{\eff}\leq \eps^2 C(\Norm{q}_{W^{1,\infty}(\R\times\S)})$, and Theorem~\ref{th.comparison-eigenvalues} follows.
\end{proof}

\subsection{Application}\label{sec.eigenvalues-regimes}

In this section, we obtain the asymptotic behavior of the low-lying spectrum of the operator $\L$. We showed in Theorem~\ref{th.comparison-eigenvalues} that the eigenvalues can be compared with the ones of the effective operator,
\[\L^\eff:=D_x^2+\eps^2V_0(\eps^\alpha x)+\eps^{3+\alpha}V_1(\eps^\alpha x),\]
where we recall that
\[V_0(X)=-\int_{\S}|\partial_{y}Q(X, y)|^2\dx y\quad \text{ and } \quad V_1(X)=2 \int_\S \big((\partial_x Q)(\partial_y Q)\big)(X,y)\dx y\] with $Q$ the unique solution to 
\[ D_{y}^2Q(X,y)=-q(X,y)\,, \quad Q(X,\cdot)\in L^2(\S)\,, \quad \int_0^1 Q(X,y)\ \dx y=0\,.\]
As previously mentioned, the asymptotic behavior of the low-lying spectrum of $\L^\eff$ strongly depends on the value of $\alpha$, and we detail below the different regimes corresponding to different values of $\alpha$.

\begin{proposition}[Semiclassical regime]\label{prop.semiclassical-eigenvalues}
Let $\alpha\in(1,3)$ and $N\in\mathbb{N}$. Assume that $X\mapsto V_0(X)$ has a unique non-degenerate minimum at $X=0$. Then there exists $\eps_0>0$, such that if $\eps\in(0,\eps_0)$, then $\L$ has at least $N$ negative eigenvalues, $\lambda_{1,\eps,\alpha}<\dots<\lambda_{N,\eps,\alpha}$,
satisfying
\[
\lambda_{n,\eps,\alpha}\ = \ \eps^2 V_0(0) +\eps^{1+\alpha}(2n-1)\sqrt{\frac{V_0''(0)}2} \ + \ \O( \eps^{\min\{4,2\alpha\}})\,.
\]
If it exists, any other negative eigenvalue satisfies $\tilde\lambda_{\eps,\alpha}\geq \eps^2 V_0(0) +\eps^{1+\alpha}(2N)\sqrt{\frac12V_0''(0)} $.
\end{proposition}
\begin{proof} Because $\alpha>1$, all the previous results (and in particular Theorem~\ref{th.comparison-eigenvalues}) hold immediately and without loss of precision when replacing the operator $\L^{\eff}$ with the simpler $ \L^{\eff,0}:=D_x^2+\eps^2V_0(\eps^\alpha x)$.
By a rescaling argument, $\lambda^{\eff,0}_{\eps,\alpha}$ is an eigenvalue of the effective operator, $\L^{\eff,0}$, if and only if $\eps^{-2}\lambda^{\eff,0}_{\eps,\alpha}$ is an eigenvalue of
\[ \L^{\sc}:=\eps^{2(\alpha-1)} D_x^2+ V_0.\]
Thus (see classical references \cite{Sim83, CFKS87, Hel88} for instance), as $h=\eps^{\alpha-1} \to0$, one has
\[ \eps^{-2}\lambda^{\eff,0}_{\eps,\alpha}=V_0(0)+(2n-1)\eps^{\alpha-1}\sqrt{\frac{V_0''(0)}{2}}+\O(\eps^{2(\alpha-1)} )\,.\]
The result now follows from Theorem~\ref{th.comparison-eigenvalues}, since the restriction $\alpha\in(1,3)$ ensures that $1+\alpha<\min\{4,4(1+\alpha),3+\alpha,2\alpha\}=\min\{4,2\alpha\}$.
\end{proof}

\begin{proposition}[Weak coupling regime]\label{prop.small-amplitude-eigenvalues}
Let $\alpha\in(0,1)$, and assume that $V_0$ is not almost everywhere zero and such that $(1+|\cdot|)V_0,(1+|\cdot|)V_1\in L^1$. Then there exists $\eps_0>0$ such that for any $\eps\in(0,\eps_0)$, $\L$ has a negative eigenvalue, $\lambda_{\eps,\alpha}$, satisfying
\[ \lambda_{\eps,\alpha}=-\frac14\eps^{4-2\alpha}\left(\int_\R V_0\right)^2+\O(\eps^{\min\{4,6-4\alpha\}})\,.\]
If it exists, any other negative eigenvalue satisfies $\tilde\lambda_{\eps,\alpha}=\O(\eps^4)$.
\end{proposition}
\begin{proof}
By a scaling argument, $\lambda^{\eff}_{\eps,\alpha}$ is an eigenvalue of $\L^{\eff}$ if and only if $\eps^{-2\alpha}\lambda^{\eff}_{\eps,\alpha}$ is an eigenvalue of
\[ \L^{\sa}:=D_x^2+\eps^{2(1-\alpha)} \big(V_0+\eps^{1+\alpha}V_1\big).\]
Since $\alpha<1$, $(1+|\cdot|)V_0\in L^1$, and $\int_\R V_0< 0$ by~\eqref{eq.V}, we are in the situation studied in~\cite[Theorem~2.5]{Simon76}, and ~\cite[Theorem~4]{Klaus77}. Their results do not directly apply due to the presence of the correction $\eps^{1+\alpha}V_1$, however their proofs are easily adapted to this situation. It follows that for $\eps$ sufficiently small, $\L^{\eff}$ has a unique negative eigenvalue, $\lambda^{\eff}_{\eps,\alpha}$, and
\[  \eps^{-2\alpha}\lambda^{\eff}_{\eps,\alpha}=-\frac14\eps^{4-4\alpha}\left(\int_\R V_0+\eps^{1+\alpha}V_1\right)^2+\O(\eps^{6-6\alpha})\,.\]
The result now follows from Theorem~\ref{th.comparison-eigenvalues},
since the restriction $\alpha\in(0,1)$ ensures that $4-2\alpha<\min\{4,4(1+\alpha),6-4\alpha\}$ and $(4-2\alpha)+(1+\alpha)>4$.
\end{proof}
\begin{proposition}[Critical regime]\label{prop.critical-eigenvalues}
Let $\alpha=1$, and assume that $V_0$ is not almost everywhere zero and such that $(1+|\cdot|)V_0\in L^1$. Denote
\[\lambda_1<\lambda_2< \dots< \lambda_N<0\]
the negative eigenvalues of
\[ \mathcal{L}^{\cri}:=D_x^2+ V_0.\]
Then for $\eps$ sufficiently small, $\L$ has $N$ negative eigenvalues, $\lambda_{n,\eps}$, satisfying
\[ \lambda_{n,\eps}=\eps^2\lambda_n+\O(\eps^4)\,.\]
If it exists, any other negative eigenvalue $\tilde\lambda_{\eps}$ satisfies $\tilde\lambda_{\eps}=\O(\eps^4)$.
\end{proposition}
\begin{proof}
As previously, we introduce $ \L^{\eff,0}:=D_x^2+\eps^2V_0(\eps^\alpha x)$ and notice that $\lambda^{\eff,0}_{n,\eps}$ is an eigenvalue of $\L^{\eff,0}$, if and only if $\eps^{-2\alpha}\lambda^{\eff,0}_{n,\eps}$ is an eigenvalue of $\mathcal{L}^{\cri}$. Since $\alpha<1$, $(1+|\cdot|)V_0\in L^1$, and $V_0\leq 0$ by~\eqref{eq.V}, $\mathcal{L}^{\cri}$ possesses $N\geq 1$ negative eigenvalues. The result then follows from Theorem~\ref{th.comparison-eigenvalues} and the comparison 
\[  \abs{\lambda^{\eff,0}_{n,\eps}-\lambda^{\eff}_{n,\eps}}\leq \eps^{4} C(\Norm{q}_{W^{1,\infty}(\R)}),\]
where $\lambda^{\eff}_{n,\eps}$ is the eigenvalue of $\L^\eff=D_x^2+\eps^2 V_0(\eps^\alpha x)+\eps^4V_1(\eps^\alpha x)$, counted as in Notation~\ref{def.notation}.
\end{proof}

\section{Description of the eigenfunctions}\label{sec.eigenfunctions}

This section is dedicated to the description of the eigenfunctions associated with the low-lying spectrum of our operator $\L$, as described in Theorem~\ref{th.main-result-eigenvalues}. The main tool is the transformation defined in Lemma~\ref{lem.normal-form-transform} which, as seen in Lemma~\ref{lem.normal-form}, allows to transform the oscillatory problem into a normal form, the latter being described at first order by the effective operator $\L^\eff :=D_x^2+\eps^2 V_0(\eps^\alpha x)+\eps^{3+\alpha} V_1(\eps^\alpha x)$; see Lemma~\ref{lem.normal-form-vs-effective}.

Consequently, the eigenfunctions of the oscillatory operator define quasimodes of the effective operator.
When the precision of the constructed quasimode is smaller than the spectral gap, one deduces an asymptotic description of the quasimode, and therefore of the oscillatory eigenfunction. In the following sections, we carry out this strategy in the different regimes so as to prove Propositions~\ref{prop.semiclassical},~\ref{prop.small-amplitude} and \ref{prop.critical}.

\subsection{Semiclassical regime $\alpha>1$; proof of Proposition~\ref{prop.semiclassical}} \label{sec.semiclassical}

We shall make use of the following properties on the eigenfunctions of $\L^\eff$ in the semiclassical limit. This proposition is a consequence of the harmonic approximation (see the classical references \cite{Sim83, CFKS87, Hel88}).
\begin{proposition}\label{prop.quasi-semiclassical}
Let $\alpha>1$ and assume that $X\mapsto V_0(X)$ has a unique non-degenerate minimum at $X=0$. Then there exists $C,\eps_0>0$ such that if $\eps\in(0,\eps_0)$, then there exists $\lambda^\eff_{1,\eps,\alpha}<\dots<\lambda^\eff_{N,\eps,\alpha}$ eigenvalues and $\varphi^{\eff,0}_{1,\eps,\alpha},\dots,\varphi^{\eff,0}_{N,\eps,\alpha}$ corresponding eigenfunctions of $\L^{\eff,0}=D_x^2+\eps^2 V_0(\eps^\alpha x)$. Moreover, $\varphi^{\eff,0}_{n,\eps,\alpha}$ is uniquely determined by 
\[\Norm{\varphi^{\eff,0}_{n,\eps,\alpha}}_{L^2(\R)}=1\,,\qquad\int_{\R}\varphi^{\eff,0}_{n,\eps,\alpha}(x) H_n(\eps^{\frac{1+\alpha}2}x)\dx x>0\,,\] 
and one has
\begin{equation}\label{eq.lambdaneps}
\left|\lambda^{\eff,0}_{n,\eps}\ - \ \left( \eps^2 V_0(0) +\eps^{1+\alpha}(2n-1)\sqrt{\frac{V_0''(0)}2}\right)\right| \leq C\times \eps^{2\alpha}\,,
\end{equation}
and
\begin{equation}\label{eq.phineps}
\varphi^{\eff,0}_{n,\eps,\alpha}(x)=\eps^{\frac{1+\alpha}4} \big(H_n(\eps^{\frac{1+\alpha}2}x)+r_{n,\eps}(\eps^{\frac{1+\alpha}2}x)\big)
\end{equation}
with
\[
\Norm{r_{n,\eps}}_{L^2(\mathbb{R})}\leq C\times \eps^{\frac{\alpha-1}2}\,.
\]
\end{proposition}
\begin{proof}
Let us only sketch the main steps of the proof. 
Using the rescaling $\hat x=\eps^{\alpha}x$ and denoting $h=\eps^{\alpha-1}$ the effective semiclassical parameter, the study reduces to the spectral analysis of
\[\mathcal{L}_{h}=h^2D^2_{\hat x}+V_0(\hat x)\,.\]
Since $X\mapsto V_0(X)$ has a unique non-degenerate minimum at $X=0$ that is not attained at infinity (as $V_0(X)\to0$ as $|X|\to\infty$), the standard harmonic approximation shows that, for all $n\geq 1$, there exist $C_{n}>0$ and $h_{n}>0$ such that for all $h\in(0, h_{n})$, the $n^{\rm th}$ eigenvalue of $\mathcal{L}_{h}$, denoted $\lambda_n(h)$, satisfies
\begin{equation}\label{eq.semiclass}
\left|\lambda_{n}(h)-V_0(0)-(2n-1)h\sqrt{\frac{V_0''(0)}{2}}\right|\leq C_{n}h^2\,,
\end{equation}
and the constants $C_{n},h_n$ depend only on $n$ and $\Norm{V_0}_{W^{3,\infty}(\R)}$ (and thus on $\Norm{q}_{W^{3,\infty}(\R\times\S)}$). We deduce~\eqref{eq.lambdaneps}.

We also observe that
\[\left\|\left(\mathcal{L}_{h}-V_0(0)-(2n-1)h\sqrt{\frac{V_0''(0)}{2}}\right)h^{-\frac{1}{4}}H_{n}(h^{-\frac{1}{2}}\cdot)\right\|_{L^2(\R)}\leq C_{n}h^{\frac{3}{2}}\,.\]
Since the $n^{\th}$ eigenspace is one-dimensional (and the spectral gap of order $h$), we get, by the spectral theorem, that the $n^{\th}$ normalized eigenfunction is at a distance, in $L^2$-norm, at most $C_{n}h^{\frac{1}{2}}$ of the normalized quasimode $h^{-\frac{1}{4}}H_{n}(h^{-\frac{1}{2}}\cdot)$. In other words, if $\hat r_{h}$ is the difference between the quasimode and the normalized eigenfunction, we have $\Norm{\hat r_{h}}_{L^2(\R)}\leq Ch^{\frac{1}{2}}$. After rescaling, we deduce \eqref{eq.phineps}. 
\end{proof}
We can now prove Proposition~\ref{prop.semiclassical}. Let $\psi_{n,\eps,\alpha}$ be the normalized eigenfunction associated with $\lambda_{n,\eps,\alpha}$, eigenvalue of $\L$, as defined by Proposition~\ref{prop.semiclassical-eigenvalues}. By Lemma~\ref{lem.normal-form}, it follows that $\varphi_{n,\eps,\alpha}:=\T^{-1}(\psi_{n,\eps,\alpha})$ satisfies
\[ \big(D_x^2+V^{\red}_{\eps,\alpha}(x)\big)\varphi_{n,\eps,\alpha}(\tilde x)=\lambda_{n,\eps,\alpha}e^{4\phi_{\eps,\alpha}(x)}\varphi_{n,\eps,\alpha}(\tilde x).\]
By Lemmata~\ref{lem.phi} and~\ref{lem.normal-form-vs-effective}, Theorem~\ref{th.comparison-eigenvalues}, Proposition~\ref{prop.semiclassical-eigenvalues} and since $\alpha>1$, one deduces
\[\Norm{(\L^\eff-\lambda^{\eff}_{n,\eps,\alpha})\varphi_{n,\eps,\alpha}}_{L^2(\R)}\leq \eps^4 C(\Norm{q}_{W^{4,\infty}(\R\times\S)},\Norm{\langle \cdot\rangle V_0'}_{L^\infty(\R)}) \Norm{\varphi_{n,\eps,\alpha}}_{L^2(\R)}\,.\]
The spectral gap is of order $\eps^{\alpha+1}$ and thus, for $\alpha\in(1,3)$, the spectral theorem yields
\[\Norm{\varphi_{n,\eps,\alpha}-\varphi^{\eff,0}_{n,\eps,\alpha}}_{L^2(\R)}\leq \eps^{3-\alpha} C\Norm{\varphi_{n,\eps,\alpha}}_{L^2(\R)}\,.\]
Proposition~\ref{prop.semiclassical} now follows from Lemma \ref{lem.normal-form-transform} and Proposition \ref{prop.quasi-semiclassical}.

\subsection{Weak coupling regime $\alpha<1$; proof of Proposition~\ref{prop.small-amplitude}} \label{sec.small-amplitude}
We shall make use of the following properties on the eigenfunctions of $\L^\eff$ in the weak coupling limit. 
\begin{proposition}\label{prop.quasi-small-amplitude}
Let $\alpha<1$, and assume that $X\mapsto V_0(X)$ is not almost everywhere zero and satisfies the integrability condition $(1+|\cdot|)V_0,(1+|\cdot|)V_1\in L^1(\R)$. Then there exists $C,\eps_0>0$ such that if $\eps\in(0,\eps_0)$, $\L^{\eff}=D_x^2+\eps^2 V_0(\eps^\alpha x)+\eps^{3+\alpha} V_1(\eps^\alpha x)$ has a unique eigenvalue denoted $\lambda^\eff_{\eps,\alpha}<0$. The corresponding eigenfunction, $\varphi^\eff_{\eps,\alpha}$, is uniquely determined by 
\[\Norm{\varphi^\eff_{\eps,\alpha}}_{L^2(\R)}=1\,,\qquad\int_{\R}\varphi^\eff_{\eps,\alpha}(x) \dx x>0\,,\] 
and one has
\begin{equation}\label{eq.lambdaeps}
\left|\lambda^\eff_{\eps,\alpha}+\frac14\eps^{4-2\alpha}\left(\int_\R V_0\right)^2\right| \leq C\times \eps^{\min\{6-4\alpha,5-\alpha\}}\,,
\end{equation}
and
\begin{equation}\label{eq.phieps}
\varphi^\eff_{\eps,\alpha}(x)= \left(\frac{\eps^{2-\alpha}}2 \int_{\R}|V_0|\right)^{\frac{1}{2}} \left(\exp \Big(|x|\frac{\eps^{2-\alpha}}2\int_{\R}V_0\Big) +r_{\eps,\alpha}(\eps^{2-\alpha}x)\right)
\end{equation}
with
\[
\Norm{r_{\eps,\alpha}}_{L^2(\mathbb{R})}\leq C\times \eps^{\min\{\frac{4}{3}(1-\alpha),1+\alpha\}}\,.
\]
\end{proposition}
\begin{proof}By rescaling, $(\lambda_{\eps,\alpha}^\eff,\varphi_{\eps,\alpha}^\eff)$ is an eigenmode of the operator $\L^\eff$ if and only if $(\eps^{-2\alpha}\lambda^\eff_{\eps,\alpha},\varphi_{\eps,\alpha}^\eff(\eps^{-\alpha}\cdot))$ is an eigenmode of 
\[ \L^{\sa}:=D_x^2+\eps^{2(1-\alpha)}( V_0+\eps^{1+\alpha} V_1).\]
 Without the correction term $V_1$, the existence and uniqueness for $\eps^{2(1-\alpha)} $ sufficiently small of a negative eigenvalue (since $V_0$ is real-valued, has negative mass and satisfies the integrability condition) as well as its asymptotic behavior as $\eps^{2(1-\alpha)} \to 0$, yielding~\eqref{eq.lambdaeps}, is a classical result of Simon~\cite{Simon76} and Klaus~\cite{Klaus77} (as mentioned in the proof of Proposition~\ref{prop.small-amplitude-eigenvalues}, the proof is easily adapted to the presence of $V_1$). As far as we know, the corresponding eigenfunction asymptotic has been first described in~\cite{ZM03}, but their result is restricted to smooth, compactly supported potentials. A less precise estimate was given in~\cite[Theorem~3.1]{DucheneVukicevicWeinstein15}, namely
\[ \sup_{x\in\R } \left\vert \varphi_{\eps,\alpha}^\eff(\eps^{-\alpha}x) -K\exp \Big(|x|\frac{\eps^{2(1-\alpha)} }2\int_{\R}V_0\Big) \right\vert = \O( \eps^{1-\alpha} )\,,\]
with renormalization constant $K\in\R$ (and, again, with $V_1=0$).
We prove below a variant of this estimate, which allows correction terms and most importantly control the $L^2$-norm.

Define $\underline x=\eps^{2-\alpha} x$, $\lambda^\eff_{\eps,\alpha}=:-\eps^{4-2\alpha}\theta^2$ and $\varphi^\eff_{\eps,\alpha}=:\eps^{1-\alpha/2}\varphi(\eps^{2-\alpha} x)$, so that
\[\big(D_{\underline x}^2+\theta^2\big) \varphi(\underline x)=-\delta^{-1} V_\eps(\delta^{-1}\underline x)\varphi(\underline x)\,,\]
where $\delta=\eps^{2-2\alpha}$ is a small parameter, and $V_\eps=V_0+\eps^{1+\alpha} V_1$. Applying the Fourier transform, we find
\begin{equation}\label{eq.eve-F}
(4\pi^2|\cdot|^2+\theta^2)\widehat\varphi=-\widehat{V_\eps}(\delta\cdot)\star\widehat{\varphi}\,,
\end{equation}
where the Fourier transform of a function $f$ is defined by the formula
\[\forall\xi\in\R\,,\quad\widehat f(\xi)=\int_{\R} e^{-2i\pi x} f(x)\dx x\,.\]
Then, we decompose the solution of \eqref{eq.eve-F} in terms of small and large frequencies
\[\widehat{\varphi}=\widehat{\varphi}_{\rm small}+\widehat{\varphi}_{\rm large}=\chi(|\cdot|\leq \delta^{-r}) \times \widehat{\varphi}+\chi(|\cdot|> \delta^{-r})\times \widehat{\varphi}\,,\]
with $\chi(S)$ the characteristic function of the set $S$, and $r>0$ is a parameter, to be determined.
With these notations, \eqref{eq.eve-F} implies that
\begin{align}\label{eqn.far}
\widehat{\varphi}_{\rm large}=\frac{\chi(|\cdot|> \delta^{-r})}{4\pi^2|\cdot|^2+\theta^2}\widehat V_\eps(\delta\cdot)\star \big(\widehat{\varphi}_{\rm small}+\widehat{\varphi}_{\rm large}\big)\,,\\
\label{eqn.near}
 \widehat{\varphi}_{\rm small}=\frac{\chi(|\cdot|\leq \delta^{-r})}{4\pi^2|\cdot|^2+\theta^2}\widehat V_\eps(\delta\cdot)\star \big(\widehat{\varphi}_{\rm small}+\widehat{\varphi}_{\rm large}\big)\,.
\end{align}
One easily checks that the operator
\[ \mathcal T:\widehat f\mapsto \frac{\chi(|\cdot|> \delta^{-r})}{4\pi^2|\cdot|^2+\theta^2}\widehat V_\eps(\delta\cdot)\star\widehat f\]
is bounded as an operator from $L^1(\R)$ to $L^1(\R)$. Moreover, one has
\[ \Norm{\mathcal T \widehat f}_{L^1(\R)}\leq \left\|\frac{\chi(|\cdot|> \delta^{-r})}{4\pi^2|\cdot|^2+\theta^2}\right\|_{L^1(\R)}\Norm{\widehat V_\eps(\delta\cdot)\star\widehat f}_{L^\infty(\R)}\leq \delta^{r} C(\Norm{\widehat V_\eps}_{L^\infty(\R)})\Norm{\widehat f}_{L^1(\R)}\,.\]
It follows that, provided $r>0$ and $\delta$ is chosen sufficiently small,~\eqref{eqn.far} defines uniquely $\widehat{\varphi}_{\rm large}$, and we get the following rough microlocalization estimate:
\begin{equation}\label{eq.microloc}
\Norm{\widehat{\varphi}_{\rm large}}_{L^1(\R)}\leq  \delta^{r}C(\Norm{\widehat V_\eps}_{L^\infty(\R)})\Norm{\widehat{\varphi}_{\rm small}}_{L^1(\R)}\,.
\end{equation}
Now, by \eqref{eqn.near}, we get
\begin{equation}\label{eq.approx-f-near}
(4\pi^2|\cdot|^2+\theta^2)\widehat{\varphi}_{\rm small}= \chi(|\cdot|\leq \delta^{-r})  \widehat V_\eps(0) \int_{\R} \chi(|\eta|\leq \delta^{-r}) \widehat{\varphi}_{\rm small}(\eta)\dx\eta \ + R_{I}+R_{II}\,, 
\end{equation}
where
\[R_{I}:=\chi(|\cdot|\leq \delta^{-r})\widehat V_\eps(\delta\cdot)\star \widehat{\varphi}_{\rm large}\,,\quad R_{II}:=\chi(|\cdot|\leq \delta^{-r})\big(\widehat V_\eps(\delta\cdot)-\widehat V_\eps(0)\big)\star \widehat{\varphi}_{\rm small}\,.\]
We estimate below the two remainders. By Young's inequality, one gets
\[\Norm{(1+|\cdot|)^{-1}R_I}_{L^2(\R)}\leq  \Norm{\frac{\chi(|\cdot|\leq \delta^{-r})}{1+|\cdot|}}_{L^2(\R)}\Norm{\widehat V_\eps(\delta\cdot )\star \widehat{\varphi}_{\rm large}}_{L^\infty(\R)}\leq C(\Norm{\widehat V_\eps}_{L^\infty(\R)})\Norm{\widehat{\varphi}_{\rm large}}_{L^1(\R)}\,,\]
and thus, with \eqref{eq.microloc},
\[\Norm{(1+|\cdot|)^{-1}R_I}_{L^2(\R)}\leq  \delta^r C(\Norm{\widehat V_\eps}_{L^\infty(\R)})\Norm{\widehat{\varphi}_{\rm small}}_{L^1(\R)}\,,\]
and by Cauchy-Schwarz inequality,
\begin{equation}\label{eq.RI}
\Norm{(1+|\cdot|)^{-1}R_I}_{L^2(\R)}\leq \delta^rC(\Norm{\widehat V_\eps}_{L^\infty(\R)})\Norm{(1+|\cdot|)\widehat{\varphi}_{\rm small}}_{L^2(\R)}\,.
\end{equation}
By the Taylor formula and the fact that $(1+|\cdot|)V_\eps\in L^1(\R)$, we can write
\begin{align*}
&\Norm{(1+|\cdot|)^{-1}R_{II}}_{L^2(\R)}^2\\
&\leq \int_{\R_{\xi}} \frac{\chi(|\cdot|\leq \delta^{-r})}{1+\xi^2} \left(\int_{\R_{\eta}} |\widehat V_0(\delta\xi-\delta\eta)-\widehat V_\eps(0)|\widehat{\varphi}_{\rm small}(\eta)\dx\eta\right)^2\dx\xi \\
&\leq \Norm{\widehat V_\eps'}^2_{L^\infty(\R)} \Norm{(1+|\cdot|)\widehat{\varphi}_{\rm small}}_{L^2(\R)}^2\int_{\R^2} \frac{\delta^2|\xi-\eta|^2}{(1+\xi^2)(1+\eta^2)} \chi(|\xi|\leq \delta^{-r})\chi(|\eta|\leq \delta^{-r})\dx\eta\dx\xi\,.
\end{align*}
From elementary considerations to estimate the last integral, we deduce that
\begin{equation}\label{eq.RII}
\Norm{(1+|\cdot|)^{-1}R_{II}}_{L^2(\R)}\leq \delta^{1-\frac{r}{2}} C(\Norm{\widehat V_\eps'}_{L^\infty(\R)}) \Norm{(1+|\cdot|)\widehat{\varphi}_{\rm small}}_{L^2(\R)}\,.
\end{equation}
Combining \eqref{eq.RI} and \eqref{eq.RII}, we are led to take $r=\frac{2}{3}$ and we get
\[\Norm{(1+|\cdot|)^{-1}(R_{I}+R_{II})}_{L^2(\R)} \lesssim \delta^{\frac{2}{3}} C(\Norm{\widehat V_\eps}_{W^{1,\infty}(\R)}) \Norm{(1+|\cdot|)\widehat{\varphi}_{\rm small}}_{L^2(\R)}\,.\]
Coming back to \eqref{eq.approx-f-near} and using \cite[Lemma 4.4]{DucheneVukicevicWeinstein15}, we find that there exists $K>0$ such that
\begin{equation}\label{eq.approx0}
\left\|(1+|\cdot|)\left(\widehat{\varphi}_{\rm small}-K\frac{\chi(|\cdot|\leq \delta^{-r})}{4\pi^2|\cdot|^2+\theta_0^2}\right)\right\|_{L^{2}(\R)}\leq \delta^{\frac{2}{3}}C(\Norm{\widehat V_\eps}_{W^{1,\infty}(\R)})\,,
\end{equation}
where we denote $\theta_0=\frac12\abs{\widehat V_\eps(0)}$.

Let us notice that, by \eqref{eqn.far} and Young's inequality for the convolution,
\[\Norm{\widehat{\varphi}_{\rm large}}_{L^{2}(\R)}\leq  \delta^{\frac{3r}{2}}C(\Norm{\widehat{V_\eps}}_{L^\infty(\R)})\ \big(\Norm{\widehat{\varphi}_{\rm small}}_{L^1(\R)}+\Norm{\widehat{\varphi}_{\rm large}}_{L^1(\R)}\big)\,.\]
Then, we notice, from the definition of $\widehat{\varphi}_{\rm small}$, Cauchy-Schwarz inequality and Plancherel's theorem, that
\[\Norm{\widehat{\varphi}_{\rm small}}_{L^1(\R)}\leq  \delta^{-\frac{r}{2}} \Norm{\widehat{\varphi}_{\rm small}}_{L^2(\R)}\leq \delta^{-\frac{r}{2}}\,.\]
Thus by the above and~\eqref{eq.microloc}, one obtains
\[\Norm{\widehat{\varphi}_{\rm large}}_{L^2(\R)}\leq  \delta^{r} C(\Norm{\widehat{V_\eps}}_{L^\infty(\R)})\,.\]
It is now easy to deduce from \eqref{eq.approx0} that $\varphi$, the solution to~\eqref{eq.eve-F}, satisfies
\[\left\|\widehat\varphi-K\frac{1}{4\pi^2|\cdot|^2+\theta_0^2}\right\|_{L^2(\R)}\lesssim \delta^{\frac{2}{3}}=\eps^{\frac{4}{3}(1-\alpha)}\,.\]
Estimate~\eqref{eq.phieps} follows by using the inverse Fourier transform, while the value of the constant, $K$, is determined by the normalization of $\varphi^\eff_{\eps,\alpha}$. We can then replace $V_\eps$ by $V_0$ in~\eqref{eq.lambdaeps} and~\eqref{eq.phieps} in the formula, up to straightforwardly estimated terms. Proposition~\ref{prop.quasi-small-amplitude} is proved.
\end{proof}

We prove Proposition~\ref{prop.small-amplitude} as in the previous section. Let $(\lambda_{\eps,\alpha},\psi_{n,\eps,\alpha})$ be the eigenmode of $\L$ uniquely defined by Proposition~\ref{prop.small-amplitude-eigenvalues}. By Lemmata~\ref{lem.phi},~\ref{lem.normal-form}, and~\ref{lem.normal-form-vs-effective}, Theorem~\ref{th.comparison-eigenvalues} and Proposition~\ref{prop.small-amplitude-eigenvalues}, and since $\alpha\in(0,1)$, $\varphi_{\eps,\alpha}:=\T^{-1}(\psi_{\eps,\alpha})$ satisfies
\[\Norm{(\L^\eff-\lambda^\eff_{\eps,\alpha})\varphi_{\eps,\alpha}}_{L^2(\R)}\leq  \eps^4 C(\Norm{q}_{W^{4,\infty}(\R\times\S)},\Norm{\langle \cdot\rangle V_0'}_{L^\infty(\R)}) \Norm{\varphi_{\eps,\alpha}}_{L^2(\R)}\,.\]
The spectral gap is of order $\eps^{4-2\alpha}$ and thus, by the spectral theorem, for $\alpha\in(0,1)$,
\[\Norm{\varphi_{\eps,\alpha}-\varphi^\eff_{\eps,\alpha}}_{L^2(\R)}\leq C\eps^{2\alpha}\Norm{\varphi_{\eps,\alpha}}_{L^2(\R)}\,.\]
Proposition~\ref{prop.small-amplitude} now follows from Lemma \ref{lem.normal-form-transform} and Proposition \ref{prop.quasi-small-amplitude}.

\subsection{Critical regime $\alpha=1$; proof of Proposition~\ref{prop.critical}} \label{sec.critical}
The proof in the case $\alpha=1$ is that same as in the previous two sections. The eigenmodes of the effective operator correspond to the ones of the operator $D_x^2+V_0$ after a straightforward rescaling. Proposition~\ref{prop.critical-eigenvalues} allows to compare the corresponding eigenfunction to the ones of our original operator, $\L$, as above. We leave the details to the reader.

\section{A WKB expansion}\label{sec.WKB}

As already mentioned, the precision of the estimates in the preceding section is insufficient to exhibit the fine multiscale structure ({\em i.e.} the small-amplitude oscillations) of the eigenfunctions, and only the large-scale behavior is captured. This is due to the fact that we cannot improve --- at least following our method --- the precision of the constructed quasimode, and that the smallness of the distance between two consecutive eigenvalues considerably deteriorates the effectiveness of the resolvent bound given by the spectral theorem. By contrast, many semiclassical studies offer asymptotic expansions up to arbitrary high order, through two-scale or Wentzel-Kramers-Brillouin (WKB) expansions. At this point it is interesting to notice that two-scale expansions are hopeless in the regime $\alpha<1$ since the numerical eigenfunction displays a three-scale structure. Indeed, the amplitude of oscillations vanish outside an interval of size $\O(\eps^{-\alpha})$ --- the support of the potential --- whereas the support of the eigenfunction is of size $\O(\eps^{2-\alpha})$. Such is not the case in the regime $\alpha>1$ and one may hope that the variations of the oscillating structure, due to the variation of the oscillating potential, are small to any algebraic order thanks to the smaller support and exponential decay of the eigenfunctions. Unfortunately, we have not been able to implement two-scale or WKB expansions for any value of $\alpha>1$, but only for specific values in the countable set $\alpha\in\{1+4/k,k\in\NN^\star\}$, for reasons which become clear below. We present in this section the detailed calculations when $\alpha=2$.

\subsection{Trace of an operator in higher dimension}

We seek to construct quasimodes of the operator $\L$ with two-scale feature
\[
 (\L-\lambda_{\eps,\alpha,N})\varphi_{\eps,\alpha,N}(x) := \left( D_x^2+q(\eps^\alpha x,x/\eps)-\lambda_{\eps,\alpha,N}\right)\varphi_{\eps,\alpha,N}(x) \ =\ \O(\eps^N)\,, 
\]
where $N$ is arbitrary large, $ \lambda_{\eps,\alpha,N}\in\R$ and $ \varphi_{\eps,\alpha,N}(x)=\Psi_{\eps,N}(\eps^\alpha x,x/\eps)$. For this to hold, we construct $\Psi_{\eps,N}:(X,y)\in \R\times \S\mapsto \Psi_{\eps,N}(X,y)$ as a quasimode of a two-dimensional operator: 
\begin{equation}\label{quasi.Xy}
\left( (\eps^{\alpha}D_X+\eps^{-1} D_y)^2+q(X,y)-\lambda\right)\Psi_{\eps,N}(X,y)\ =\ \O(\eps^N)\,.
\end{equation}
Denoting $h=\eps^{\alpha-1}$, we find
\[
 \left( (h D_X+h^{\frac{-2}{\alpha-1}} D_y)^2+ h^{\frac{-2}{\alpha-1}} q(X,y)-h^{\frac{-2}{\alpha-1}}\lambda \right)\Psi_{\eps,N}(X,y)\ =\ \O(h^{\frac{N-2}{\alpha-1}})\,.
\]
A two-scale expansion would require a further scaling $\tilde X=h^{1/2}X$, and one readily sees that the size of the differential operators as well as the Taylor expansion of $h^{\frac{-2}{\alpha-1}} q(h^{1/2}\tilde X,y)$ around $X=0$ are all powers of $h^{1/2}$ provided that $\frac{4}{\alpha-1}\in\NN$. In the following, we limit ourselves to the specific value $\alpha=2$, and present the more efficient WKB expansion.

\subsection{The formal expansion}

Summarizing the above and abusing notations, we seek $\lambda_{h,N}$ and $\Psi_{h,N}$ satisfying
\begin{equation}\label{BKW}
 (\LL_h-\lambda_{h,N})\Psi_{h,N}:=\left( (h^3 D_X+ D_y)^2+ h^2 q(X,y)-\lambda_{h,N} \right)\Psi_{h,N}\ =\ \O(h^{N})\,,
\end{equation}
with $N\in\NN$ arbitrary large. In order to do so, we introduce a (real-valued) phase $\Phi(X)$ and notice that
\[\LL_h^\Phi:=e^{\Phi/h} \LL_h e^{-\Phi/h} = (h^3 D_X+ih^2\Phi'(X) +D_y)^2+ h^2 q(X,y)\]
reads
\[\LL_h^\Phi=\sum_{k=0}^6 h^k \LL_k\]
with
\[\LL_0=D_y^2, \quad \LL_1=0, \quad \LL_2=2\Phi'(X)\partial_y+q(X,y), \quad \LL_3=-2\partial_X\partial_y, \quad \LL_4=-\Phi'(X)^2, \]
\[ \LL_5=\partial_X\Phi'(X)-\Phi'(X)\partial_X, \qquad \LL_6=D_X^2.\]
We seek $\Psi_{h,N}$ and $\lambda_{h,N}$ as power expansions
\[\Psi_{h,N}=\sum_{k=0}^{N-1} h^k \Psi_k(X,y) \quad ; \quad  \lambda_{h,N}=\sum_{k=0}^{N-1} h^k \lambda_k\,,\]
and find $\Psi_{h,N},\lambda_{h,N}$ by plugging the above expansions into~\eqref{BKW} and solving at maximal order.

\subsubsection{Initialization}

We determine the explicit formulas for the first-order contributions in $\Psi_{h,N},\lambda_{h,N}$ by solving~\eqref{BKW} up to the order $\O(h^6)$.

\paragraph{Order $\O(h^0)$.}
We need to solve $D_y^2\Psi_0=\lambda_0\Psi_0$, and deduce from the Fredholm alternative (in the variable $y$) that
\begin{equation}\label{lambda0-Psi0y}
\lambda_0=0 \quad \text{ and } \quad \Psi_0(X,y)=f_0(X)\,,
\end{equation}
where $f_0$ will be determined later on.

\paragraph{Order $\O(h^1)$.}
We need to solve, after using~\eqref{lambda0-Psi0y}, $D_y^2\Psi_1=\lambda_1\Psi_0$, from which we deduce as above
\begin{equation}\label{lambda1-Psi1y}
\lambda_1=0 \quad \text{ and } \quad \Psi_1(X,y)=f_1(X)\,,
\end{equation}
where $f_1$ will be determined later on.

\paragraph{Order $\O(h^2)$.}
We need to solve
\[ D_y^2\Psi_2(X,y)=\big(\lambda_2\Psi_0-\LL_2\Psi_0\big)(X,y)=\lambda_2-f_0(X)q(X,y),\]
and the Fredholm alternative (in the variable $y$), using that $q(X,\cdot)$ is mean-zero, yields
\begin{equation}\label{lambda2-Psi2y}
\lambda_2=0 \quad \text{ and } \quad \Psi_2(X,y)=f_0(X)Q(X,y)+f_2(X)\,,
\end{equation}
where $f_2$ will be determined later on and denote as always $Q$ as the unique solution (for any fixed $X\in\R$) to
\[D_y^2Q(X,y)=-q(X,y), \quad \int_\S Q(X,y)\dx y=0.\]

\paragraph{Order $\O(h^3)$.}
We need to solve
\[ D_y^2\Psi_3(X,y)=\big((\lambda_3-\LL_3)\Psi_0+(\lambda_2-\LL_2)\Psi_1\big)(X,y)=\lambda_3 f_0(X)-f_1(X)q(X,y),\]
from which we deduce as above
\begin{equation}\label{lambda3-Psi3y}
\lambda_3=0 \quad \text{ and } \quad \Psi_3(X,y)=f_1(X)Q(X,y)+f_3(X)\,,
\end{equation}
where $f_3$ will be determined later on.

\paragraph{Order $\O(h^4)$.}
We need to solve
\begin{align*}
 D_y^2\Psi_4(X,y)&=\big((\lambda_4-\LL_4)\Psi_0+(\lambda_3-\LL_3)\Psi_1+(\lambda_2-\LL_2)\Psi_1\big)(X,y)\\
 &=(\lambda_4+\Phi'(X)^2) f_0(X)-2\Phi'(X)f_0(X)\partial_yQ(X,y)-f_0(X)q(X,y)Q(X,y)\\
 &\quad -f_2(X)q(X,y)\,.
 \end{align*}
Using again that $q(X,\cdot)$ is mean-zero, the Fredholm alternative (in the variable $y$) yields the eikonal equation
\[\lambda_4+\Phi'(X)^2-\int_{\S}q(X,y)Q(X,y)\dx y=0\,.\]
Here, we recognize, after one integration by parts,
\[\int_{\S}q(X,y)Q(X,y)\dx y=V_{0}(X)\]
as defined in~\eqref{eq.V}. Here and below, we shall assume that $V_0$ has a unique non-degenerate minimum, at $X=0$. In order for $\Phi$ to be a smooth non-negative solution to the above, one needs 
\begin{equation}\label{lambda4}
\lambda_4=V_{0}(0)
\end{equation}
and
\begin{equation}\label{phi}
\Phi(X)=\left|\int_0^X\sqrt{V_{0}(s)-V_{0}(0)}\dx s\right|\,,
\end{equation}
the ``Agmon distance'' from $X$ to $0$.
With this choice, we may set
\[\psi_4(X,y)=f_4(X)+f_2(X) Q(X,y) +F_4[f_0](X,y)\,,\]
where $f_4$ will be determined later on, and $F_4$ is the unique solution to
\[D_y^2F_4(X,y)=f_0(X)\left(V_{0}(0)+\Phi'(X)^2-q(X,y)Q(X,y)\right), \quad \int_{\S}F_4(X,y)\dx y=0\,.\]

\paragraph{Order $\O(h^5)$.}
We need to solve
\begin{align*}
 D_y^2\Psi_5(X,y)&=\big((\lambda_5-\LL_5)\Psi_0+(\lambda_4-\LL_4)\Psi_1+(\lambda_3-\LL_3)\Psi_2+(\lambda_2-\LL_2)\Psi_3\big)(X,y)\\
 &=\lambda_5f_0(X)-\Phi''(X)f_0(X)-2\Phi'(X)f_0'(X)+(V_{0}(0)+\Phi'(X)^2) f_1(X)\\
 & \quad +2f_0'(X)\partial_yQ(X,y)+2f_0(X)\partial_X\partial_yQ(X,y)-2\Phi'(X)f_1(X)\partial_yQ(X,y)\\
 &\quad  -f_1(X)q(X,y)Q(X,y)-f_3(X)q(X,y)\,.
 \end{align*}
The Fredholm alternative (in the variable $y$) yields, similarly as above,
\[ \lambda_5f_0(X)-\Phi''(X)f_0(X)-2\Phi'(X)f_0'(X)+(V_{0}(0)+\Phi'(X)^2-V_{0}(X)) f_1(X)=0\,.\]
Since $\Phi$ has been constructed so as to satisfy the eikonal equation, the above simplifies to the transport equation
\[\lambda_5f_0(X)-\Phi''(X)f_0(X)-2\Phi'(X)f_0'(X)=0\,.\]
In order to resolve the singularity at $X=0$ and allow for smooth solutions $f_0$, we need to set
\begin{equation}\label{lambda5}
\lambda_5=(2n+1)\times\Phi''(0)=(2n+1)\times\sqrt{\frac12V_{0}''(0)}\,, 
\end{equation}
where $n\in\NN$ is a free parameter, and then 
\begin{equation}\label{f0}
f_0(X)=C_0X^n \exp\left(-\int_0^X\frac{\Phi''(s)-\lambda_5}{2\Phi'(s)}-\frac{n}{s}\dx s\right)\,,
\end{equation}
with $C_0\neq 0$ any multiplicative constant. Then one has
\[\psi_5(X,y)=f_5(X)+f_3(X)Q(X,y) +F_5[f_0,f_1](X,y)\,,\]
where $f_5$ will be determined later on, and $F_5$ is the unique solution to
\begin{multline*}
D_y^2F_5(X,y)=f_1(X)\left(V(0)+\Phi'(X)^2-q(X,y)Q(X,y)\right)\\
+2f_0'(X)\partial_yQ(X,y)+2f_0(X)\partial_X\partial_yQ(X,y)-2\Phi'(X)f_1(X)\partial_yQ(X,y)\end{multline*}
satisfying  $\int_{\S}F_5(X,y)\dx y=0$.

\subsubsection{Induction}

Based on the previous calculations, we guess the following Ansatz:
\begin{equation}\label{Ansatz}
\Psi_k(X,y)=f_{k}(X)+Q(X,y)f_{k-2}(X)+F_k[f_0,\dots,f_{k-4}](X,y)\,,
\end{equation}
where $F_k(X,y)$ is uniquely determined by $f_i$ (and derivatives) for $i\leq k-4$ and is mean-zero (in the variable $y$) for any value of $X\in\R$.

The form~\eqref{Ansatz} has been verified above for $k\in \{0,1,\dots,5\}$ (with the convention $f_i= 0$ for $i<0$) with explicit formulas for $\lambda_k,F_k$ and $f_0$. We explain below how $\lambda_k,F_k,f_{k-5}$ may be determined for any $k\geq 6$ by induction on $k$. For $k\geq 6$, solving~\eqref{BKW} at the order $\O(h^k)$ yields
\begin{multline*} (\LL_0-\lambda_0) \Psi_k=\sum_{i=7}^{k}\lambda_i\Psi_{k-i} +(\lambda_6-\LL_6)\Psi_{k-6}+(\lambda_5-\LL_5)\Psi_{k-5}+(\lambda_4-\LL_4)\Psi_{k-4}\\
+(\lambda_3-\LL_3)\Psi_{k-3}+(\lambda_2-\LL_2)\Psi_{k-2}+(\lambda_1-\LL_1)\Psi_{k-1}\,.
\end{multline*}
By~\eqref{lambda0-Psi0y}--\eqref{f0}, we find
\begin{align}\label{eqy} D_y^2 \Psi_k(X,y)&= \sum_{i=7}^{k}\lambda_i\Psi_{k-i}(X,y)+ (\lambda_6-D_X^2)\Psi_{k-6}(X,y)\\
&\quad+((2n+1)\Phi''(0)-\Phi''(X)-2\Phi'(X)\partial_X)\Psi_{k-5}(X,y)\nonumber\\
&\quad+(V_{0}(0)+\Phi'(X)^2)\Psi_{k-4}(X,y)+2\partial_y\partial_x\Psi_{k-3}(X,y)\nonumber\\
&\quad-2\Phi'(X)\partial_y\Psi_{k-2}(X,y)+q(X,y)\Psi_{k-2}(X,y)\,.\nonumber
\end{align}
Using~\eqref{Ansatz} for $i=\in\{0,\dots, k-1\}$, the Fredholm alternative (in the variable $y$) yields
\begin{multline*} \sum_{i=6}^{k-1}\lambda_i f_{k-i}(X)+\lambda_k f_0(X)+ f_{k-6}''(X)+((2n+1)\Phi''(0)-\Phi''(X)-2\Phi'(X)\partial_X)f_{k-5}(X)\\+(V_{0}(0)+\Phi'(X)^2)f_{k-4}(X)
+V_{0}(X)f_{k-4}(X)+\int_{\S}q(X,y)F_{k-2}[f_0,\dots,f_{k-6}](X,y)\dx y=0\,.
\end{multline*}
Thanks to~\eqref{lambda4} and~\eqref{phi}, the above reduces to
\begin{equation}\label{eqX} ((2n+1)\Phi''(0)-\Phi''(X)-2\Phi'(X)\partial_X)f_{k-5}(X)+\lambda_k f_0(X)+G_k[f_0,\dots,f_{k-6}](X)=0\,,
\end{equation}
where $G_k$ is a known function depending on $f_0,\dots,f_{k-6}$ (and derivatives) as well as $\lambda_6,\dots,\lambda_{k-1}$ (if $k-1\geq 6$).
We explain below how~\eqref{eqX} determines the unknowns $f_{k-5}\in C^\infty(\R)$ and $\lambda_k$. Then, since its right-hand side is mean-zero, the equation~\eqref{eqy} defines uniquely the mean-zero solution 
\[\Psi_k(X,y)-f_k(X)=Q(X,y)f_{k-2}(X)+F_k[f_0,\dots,f_{k-4}](X,y)\,,\]
and the induction is complete.

Thus we are left with solving the equation~\eqref{eqX} for $f_{k-5}$ and $\lambda_k$. For simplicity, we rewrite
\begin{equation}\label{eqX-2}
 (\Phi''(X)+2\Phi'(X)\partial_X-(2n+1)\Phi''(0))f(X)=G(X)+\lambda f_0(X)\,,
 \end{equation}
where $f_0(X),\Phi(X)$ and $G(X)$ are given smooth functions, and $\lambda,f(X)$ are the unknowns. As a first step, we solve \eqref{eqX-2} in the sense of formal series. Taylor expanding around $X=0$ yields
\[\left(\sum_{k\geq 0} X^k\big(a_k+b_k X\partial_X\big)\right)f(X)=\sum_{k\geq 0} \big(c_k+\lambda d_k\big)X^k\,,\]
with given $a_k,b_k,c_k,d_k\in\R$; in particular, $a_0=-2n\Phi''(0)$ and $b_0=2\Phi''(0)$. Notice that, applied to the monomial $X^m$ (with $m\in\NN$), one has
\[\left(\sum_{k\geq 0} X^k\big(a_k+b_k X\partial_X\big)\right)X^m=2(m-n)\Phi''(0)X^m+\sum_{k\geq 1} (a_k+mb_k) X^{m+k}\,.\]
This allows to define uniquely $\lambda$ and $f_m$ for $m\geq 0$ such that, for arbitrary $N\in\NN$, 
\[\left(\sum_{k= 0}^N X^k\big(a_k+b_k X\partial_X\big)\right)\left(\sum_{m= 0}^N f_m X^m\right)=\sum_{k= 0}^N X^k\big(c_k+\lambda d_k\big)+\O(X^{N+1}) (X\to 0)\,,\]
Indeed, we define $f_m$ by induction on $m\geq 0$ through the identity
\[ 2(m-n)\Phi''(0)f_m +\sum_{k=1}^N (a_k+mb_k)f_{m-k}=  c_m+\lambda d_m.\]
Since $\Phi''(0)>0$, the identity is solvable for $f_m$ for any $m\neq n$. For $m=n$, we fit $\lambda$ so that the identity is satisfied ---notice that $d_n=\frac{f_0^{(n)}(0)}{n!}=C_0\neq 0$ by~\eqref{f0}--- and, for instance, $f_n=0$.

As a second step, we introduce $\tilde f\in C^\infty(\R)$ such that $\frac{\tilde f^{(m)}(0)}{m!}=f_m$, whose existence follows from Borel's Lemma. By the above, one easily checks that
\[ r(X):= (\Phi''(X)+2\Phi'(X)\partial_X-(2n+1)\Phi''(0))\tilde f(X)-G(X)-\lambda f_0(X)\]
satisfies $r^{(m)}(X)=0$ for any $m\in\NN$. We shall now find $\check f\in C^\infty(\R)$ such that
\[(\Phi''(X)+2\Phi'(X)\partial_X-(2n+1)\Phi''(0))\check f(X)=r(X)\,.\]
Indeed, let
\[\check f(X)=f_0(X)\int_0^X \frac{r(s)}{2\Phi'(s)f_0(s)}\dx s\,.\]
Notice the singularity in $X=0$ is removed by the property that $r^{(m)}(X)=0$ for any $m\in\NN$, and that $\Phi'(X)f_0(X)\neq 0$ for $X\neq 0$ by~\eqref{phi} and~\eqref{f0}. Altogether, we find that $f(X):=\tilde f(X)-\check f(X)$ satisfies~\eqref{eqX-2}, as desired.

\subsection{Completion of the proof}
The construction of $\lambda_k,\Psi_k$ above can be pursued to any arbitrary order, and allows to obtain
\[\Psi_{h,N}(X,y):=\sum_{k=0}^{N-1} h^k \Psi_k(X,y)\in C^\infty(\R\times\S) , \quad \lambda_{h,N}:=\sum_{k=0}^{N-1} h^k \lambda_k \in\R\]
such that 
\[ \big(\LL_h -\lambda_{h,N}\big) e^{-\Phi(X)/h}\Psi_{h,N}(X,y) = r_{h,N}e^{-\Phi(X)/h}\,,\]
 where $r_{h,N}\in C^\infty(\R\times\S)$ and 
 \[\Norm{r_{h,N}}_{C^k(\mathcal{U}\times\S)}=\O(h^N)\,,\]
 for all $k\in\NN$ and $\mathcal U$ bounded neighborhood of $0$.
 
 What is more, one easily checks by induction that as long as $2k\leq n$, one has
 \[ f_k(X)= C_0 \alpha_k X^{n-2k}+\O(X^{n-2k+2}), \quad \alpha_k=\frac{-\alpha_{k-1}}{\Phi''(0)} \frac{(n-2k+2)(n-2k+1)}{2k}\,.\]
 This shows that, provided $2N\geq n$, and with the right choice of constant $C_0$, 
 \[\sum_{k=0}^{N-1} h^k f_k(h^{1/2} X) =  h^{n/2} P_n(h^{1/2} X) + R_{h,n,N}(h^{1/2} X),\]
 where $P_n$ is the $n$-th rescaled Hermite polynomial associated with the solution to
 \[-H''_{n}(X)+(\Phi''(0)X)^2 H_{n}=(2n-1)\Phi''(0)H_{n}(X)\,,\]
and, for any $k\in\NN$,
 \[\Norm{R_{h,n,N}}_{C^k(\mathcal{U})}=\O(h^{n/2+1})\,,\]
  for any $\mathcal U$ bounded neighborhood of $0$. A further inspection shows that
\[  \Psi_{h,N}(h^{1/2}X,y)=\left(\sum_{k=0}^{N-1} h^k f_k(h^{1/2} X)\right)\big(1+h^2 Q(h^{1/2}X,y)\big)+F_{h,n,N}(h^{1/2} X,y)\]
with $F_{h,n,N}$ collecting higher-order contributions and satisfying
\[ \Norm{F_{h,n,N}}_{L^\infty(\mathcal{U}\times\S)}=\O(h^{n/2+3})\,,\]
for any $\mathcal U$ bounded neighborhood of $0$, and provided $2N\geq n$. Additionally, notice
\[\lambda_{h,N}=h^4\lambda_4+h^5\lambda_5+\O(h^6)=h^4 V_0(0)+h^5 (2n+1)  \Phi''(0)+\O(h^6)\,.\]

We are now in position to prove the expansions in Proposition \ref{prop.WKB}. We introduce $\chi$ a smooth cutoff function equals to $1$ in a neighborhood of $X=0$, and use $\chi(\eps^2 x)e^{-\Phi(\eps^2x)/\eps}\Psi_{\eps,N}(\eps^2x,x)$ as a quasimode for $\mathcal{L}_{\eps,2}$. Indeed, one has (see \eqref{quasi.Xy}):
\[\left(\mathcal{L}_{\eps,2}-\eps^{-2}\lambda_{\eps,N}\right)\chi(\eps^2x)e^{-\Phi(\eps^2x)/\eps}\Psi_{\eps,N}(\eps^2x, x)=\eps^{-2}\chi(\eps^2x)r_{\eps,N}(\eps^2x,x)e^{-\Phi(\eps^2x)/\eps}+\O(e^{-c/\eps})\,,\]
where $\displaystyle{c=(\inf_{\mathsf{supp}\chi'}\Phi)/2>0}$. We can then estimate the right hand side:
\[
\Norm{\chi(\eps^2x)r_{\eps,N}(\eps^2x,x)e^{-\Phi(\eps^2x)/\eps}}_{L^2(\R)}\lesssim \eps^N\Norm{\chi(\eps^2x)e^{-\Phi(\eps^2x)/\eps}}_{L^2(\R)}\,.
\]
Moreover, the previous estimate on $\Psi_{\eps,N}$ yields
\[ \Norm{\chi(\eps^2x)e^{-\Phi(\eps^2x)/\eps}\Psi_{\eps,N}(\eps^2x,x)}_{L^2(\R)}\gtrsim \eps^{n/2} \Norm{\chi(\eps^2x)e^{-\Phi(\eps^2x)/\eps}}_{L^2(\R)}\,.\]
It remains to combine the spectral theorem with the fact that, in the semiclassical case (and $\alpha=2$), the spectral gap is of order $\eps^{3}$ (by Proposition \ref{prop.semiclassical}). Proposition~\ref{prop.WKB} follows from the above estimates and defining
\[\varphi_{n,\eps}(\eps^{3/2}x)=\chi(\eps^2x)e^{-\Phi(\eps^2x)/\eps}\times \left(\sum_{k=0}^{N-1} \eps^k f_k(\eps^{2} x)\right) \,,\]
since $\Phi(X) =\frac12\Phi''(0)X^2+\O(X^3)=\frac12\sqrt{\frac{V_0''(0)}{2}}X^2+\O(X^3)$.

\appendix

\section{Numerical scheme} \label{sec.numerics}

In this section, we present the numerical scheme used in Figures~\ref{fig.alpha=2},~\ref{fig.alpha=05} and~\ref{fig.alpha=1}. Since our potential, $q$, and the expected solutions decay exponentially at infinity, it is convenient to truncate the infinite spatial domain to a periodic interval $\S(-L,L)$, and turn to Fourier spectral methods.
However, because of the several scales of our problem, it is too costly to approximate the solution to the eigenvalue problem 
\[
 \L\psi_{\eps,\alpha}(x) := \left( D_x^2+q(\eps^\alpha x,x/\eps)\right)\psi_{\eps,\alpha}(x) \ = \ \lambda_{\eps,\alpha} \psi_{\eps,\alpha}(x), \quad \psi_{\eps,\alpha}\in L^2(\mathbb{R})\,.
\]
with a complete set of Fourier modes:
\[\psi_{\eps,\alpha}(x)\approx \sum_{k=-N}^N a_k e^{i k \frac{\pi}{L} x }\,.\]
Thus we restrict to a limited number of well-chosen Fourier modes. Motivated by our results, we define
\[ \mathbb{K}_n:= \bigcup_{ j\in\{-2,-1,0,1,2\}}\{k\in\mathbb{Z}, \quad |\frac{k\pi}{L}-\frac{2\pi j}{\eps}|\leq \frac{n\pi}{L}\}\,.\]
and seek
\[\psi_{\eps,\alpha}(x)\approx \sum_{k\in \mathbb{K}_n} a_k e^{i k \frac{\pi}{L} x }.\]
In other words, defining the orthogonal projections
\begin{align*}
\Pi_N &:= f\mapsto \sum_{k=-N}^N \frac{\langle  e^{i k \frac{\pi}{L} \cdot }, f\rangle_{L^2(-L,L)}}{\langle  e^{i k \frac{\pi}{L} \cdot }, e^{i k \frac{\pi}{L} \cdot}\rangle_{L^2(-L,L)}}e^{i k \frac{\pi}{L} \cdot}\,,\\
\Pi_{\mathbb{K}_n} &:= f\mapsto \sum_{k\in\mathbb{K}_n} \frac{\langle  e^{i k \frac{\pi}{L} \cdot }, f\rangle_{L^2(-L,L)}}{\langle  e^{i k \frac{\pi}{L} \cdot }, e^{i k \frac{\pi}{L} \cdot}\rangle_{L^2(-L,L)}}e^{i k \frac{\pi}{L} \cdot}\,,
\end{align*}
we numerically solve
\[  \left( D_x^2+\Pi_{\mathbb{K}_n} \big(\Pi_N q(\eps^\alpha \cdot,\cdot/\eps)\big)\right) \widetilde\psi_{\eps,\alpha} \ = \ \widetilde\lambda_{\eps,\alpha}  \widetilde\psi_{\eps,\alpha}\,, \qquad \widetilde\psi_{\eps,\alpha}=\Pi_{\mathbb{K}_n} \widetilde\psi_{\eps,\alpha}\,,\]
as an eigenvalue problem for a matrix of size $5(2n+1)\times 5(2n+1)$.

For Figures~\ref{fig.alpha=2}~\ref{fig.alpha=05} and~\ref{fig.alpha=1}, we set $L=500$, $N=2^{22}$, and $n=300$.

\bibliographystyle{abbrv}
\def\cprime{$'$}

\end{document}